\documentclass[12pt]{amsart}
\usepackage{amssymb,amsmath,amsfonts,latexsym}
\usepackage{bm}
\usepackage[all,cmtip]{xy}
\usepackage{amscd}

\newcommand{\newsection}[1]{\setcounter{equation}{0} \section{#1}}
\newcommand{\bea}{\begin{eqnarray}}
\newcommand{\eea}{\end{eqnarray}}

\newcommand{\clb}{\mathcal{B}}

\newcommand{\clh}{\mathcal{H}}
\newcommand{\clk}{\mathcal{K}}
\newcommand{\cll}{\mathcal{L}}

\newcommand{\cln}{\mathcal{N}}

\newcommand{\clu}{\mathcal{U}}
\newcommand{\clw}{\mathcal{W}}

\newcommand{\mo}{\mathop{\oplus}}

\def\textmatrix#1&#2\\#3&#4\\{\bigl({#1 \atop #3}\ {#2 \atop #4}\bigr)}
\def\dispmatrix#1&#2\\#3&#4\\{\left({#1 \atop #3}\ {#2 \atop #4}\right)}
\newcommand{\be}{\begin{equation}}
\newcommand{\ee}{\end{equation}}
\newcommand{\ben}{\begin{eqnarray*}}
	\newcommand{\een}{\end{eqnarray*}}

\newcommand{\NI}{\noindent}

\newcommand{\bi}{\begin{itemize}}
	\newcommand{\ei}{\end{itemize}}

\theoremstyle{definition}

\theoremstyle{plain}

\newtheorem{thm}{Theorem}[section]
\newtheorem{cor}[thm]{Corollary}
\newtheorem{lem}[thm]{Lemma}
\newtheorem{prop}[thm]{Proposition}
\theoremstyle{definition}
\newtheorem{defn}[thm]{Definition}
\newtheorem{rem}[thm]{Remark}
\newtheorem{ex}[thm]{Example}

\numberwithin{equation}{section}

\let\phi=\varphi

\begin{document}

\title[On Decomposition for Pairs of Twisted Contractions]
{On Decomposition for Pairs of Twisted Contractions}

\author[Majee]{Satyabrata Majee}
\address{Indian Institute of Technology Roorkee, Department of Mathematics,
		Roorkee-247 667, Uttarakhand,  India}
\email{smajee@ma.iitr.ac.in}

\author[Maji]{Amit Maji}
\address{Indian Institute of Technology Roorkee, Department of Mathematics,
		Roorkee-247 667, Uttarakhand,  India}
\email{amit.maji@ma.iitr.ac.in, amit.iitm07@gmail.com}




\begin{abstract}
This paper presents Wold-type decomposition for various pairs 
of twisted contractions on Hilbert spaces. As a consequence,
we obtain Wold-type decomposition for pairs of doubly twisted 
isometries and in particular, new and simple proof of 
S\l{}o\'{c}inski's theorem for pairs of doubly commuting isometries 
are provided. We also achieve an explicit decomposition for pairs 
of twisted contractions such that the c.n.u. parts of the contractions 
are in $C_{00}$. It is shown that for a pair $(T,V^*)$ of twisted 
operators with $T$ as a contraction and $V$ as an isometry, there 
exists a unique (upto unitary equivalence) pair of doubly twisted 
isometries on the minimal isometric dilation space of $T$. As an 
application, we prove that pairs of twisted operators consisting of an 
isometry and a co-isometry are doubly twisted. Finally, we have 
given a characterization for pairs of doubly twisted isometries.
\end{abstract}	
\maketitle

\newsection{Introduction}

Operator theory on Hilbert spaces has two essential parts: 
the theory of normal operators and the theory of non-normal 
operators. The structure of normal operators is well known 
due to spectral decomposition. On the other hand, the complete 
structure of non-normal operators is unknown to the literature 
and an important class of non-normal operators is isometries. 
Indeed, one of the important problems in operator theory, 
analytic function theory and operator algebras is the 
classification and representation of $n$-tuples ($n > 1$) of 
commuting isometries on Hilbert spaces.
\vspace{.1cm}

In a probabilistic language, Wold \cite{WOLD-TIME SERIES} first
established a notable decomposition for stationary stochastic 
processes. Later, von Neumann, Kolmogorov, and Halmos coined the 
abstraction of Wold's result for isometries on Hilbert spaces: 
Every isometry can be uniquely written as a direct sum of a unitary 
operator and copies of the unilateral shifts. This is called {\it{Wold 
decomposition or Wold-von Neumann decomposition}} (see details in Section 2). 
This decomposition plays a vital role in many areas of operator algebras 
and operator theory, namely, dilation theory, invariant subspace theory, 
operator interpolation problem etc. It is now a natural question: 
{\textsf{Does there exist Wold-type decomposition for pairs of commuting 
isometries (in general for contractions) on Hilbert spaces?}} 
\vspace{.1cm}

There has been a lot of research in this direction for the last few decades and 
many important and interesting results have been obtained in many directions. 
Still a complete and explicit structure for pairs of commuting isometries, 
or, in general, an $n$-tuple of commuting isometries on a Hilbert space, is 
unknown. Many researchers have investigated Wold-type decomposition for a 
pair of commuting isometries/commuting partial isometries/commuting 
contractions. For instance, Suciu \cite{SUCIU-SEMIGROUPS} developed a 
structure theory for a semigroup of isometries. Later, S\l{}o\'{c}inski 
\cite{SLOCINSKI-WOLD} obtained a Wold-type decomposition for pairs of doubly 
commuting isometries from Suciu's decomposition of the semigroup of isometries. 
Burdek, Kosiek and S\l{}o\'{c}inski \cite{BKS-CANONICAL} developed the canonical Wold 
decomposition considering the finite-dimensional wandering space for commuting 
isometries. Popovici \cite{POPOVICI-WOLD TYPE} studied the Wold-type decomposition
for pairs of commuting isometries. Later, Sarkar \cite{SARKAR-WOLD} generalized 
S\l{}o\'{c}inski's result for $n$-tuple of doubly commuting isometries. Many important 
results have been obtained by many researchers, like 
Burdak, Kosiek, Pagacz and S\l{}o\'{c}inski (\cite{BKPS-COMMUTING}, 
\cite{BKPS-SHIFT TYPE}), Bercovici, Douglas and Foia\c{s} (\cite{BDF-BI-ISOMETRIES}, 
\cite{BDF-MULTI-ISOMETRIES}), Maji, Sarkar, and Sankar \cite{MSS-PAIRS}, and 
references therein etc. On the other hand, Halmos and Wallen \cite{HW-POWERS} studied
decomposition for power partial isometry in 1970. Burdak \cite{BURDAK-DECOMPOSITION} 
developed a new characterization for a pair of commuting (not necessarily doubly commuting) 
contractions and obtained decomposition results in the case of power partial isometries. 
\vspace{.1cm}

Jeu and Pinto \cite{JP-NON COMMUTING} established Wold decomposition  
for $n$-tuple $(n > 1)$ of doubly non-commuting isometries. Later, Rakshit, 
Sarkar, and Suryawanshi \cite{RSS-TWISTED ISOMETRIES} generalized those 
results by introducing $\clu_n$-twisted isometries. Recently, we introduced 
$n$-tuple of doubly non-commuting and $\clu_n$-twisted contractions (see section 
3 and section 5 in \cite{MM-Wold-type}) and then established a complete 
description of those tuples.
\vspace{0.1cm}

Dilation theory is one of the most effective tools to study the structure of 
contractions. One of the striking results is the Sz.-Nagy's dilation 
result (see \cite{NF-BOOK}): Every contraction has an isometric dilation 
on a larger Hilbert space. Later, And\^{o} \cite{ANDO-CONTRACTIONS} proved 
that a commuting pair of contractions dilates to a commuting pair of isometries. 
Other special classes of operators are $C_{\alpha \beta} = C_{\alpha \cdot} 
\cap C_{\cdot \beta}$ ($\alpha, \beta$= 0, 1), which plays a significant role 
in the study of general contractions on Hilbert spaces. Indeed, every contraction
on a Hilbert space has a canonical triangulations of the following types 
(see \cite{NF-BOOK}):
\[
	\begin{bmatrix} 
	C_{0.} & * \\ 
	O & C_{1.}
	\end{bmatrix} 
\quad
\mbox{and}
\quad
\begin{bmatrix} 
	C_{.1} & * \\ 
	O & C_{.0}
\end{bmatrix}. 
\]
The reducibility of an operator means deciding whether the operator 
has a nontrivial reducing subspace or not, and the classification of 
invariant subspaces and reducing subspaces of various operators on 
function spaces has proved to be very challenging research problems 
in analysis. The reducibility of general $C_0(N)$ operators is complicated.
However, Gu \cite{GU-reducibility} studied the reducibility of 
any power of a $C_0(1)$ operator (see also \cite{DF-On}). 
Uchiyama (\cite{U-Hyperinvariant}, \cite{U-Double}) discussed hyperinvariant 
subspaces for contractions of class $C_{.0}$ and also found double commutants for 
$C_{.0}$ contractions with finite defect indices. Recently,  Benhida et al.
\cite{BFT-Reducing} obtained several equivalent conditions for the
reducibility of a contraction in the class $C_{00}$ in terms of a 
minimal unitary dilation and the characteristic function. 
\vspace{.1cm}

The main aim of this paper is to investigate decomposition of various
pairs of twisted operators (in particular, for commuting) as well as 
the structure of different kinds of operators (contraction, isometry, 
partial isometry) on Hilbert spaces. The geometry of Hilbert spaces, 
canonical decomposition for a contraction, classical Wold-von Neumann 
decomposition for isometry and dilation theory are the essential 
tools used in this article. One of the main results is Theorem 3.1 
(see Section 3): It says that if a pair of doubly twisted operators 
$(T, V)$ on a Hilbert space $\clh$, where $T$ is a contraction, 
then the canonical decomposition for $T$ reduces $V$. Secondly is 
Theorem 5.3 (see Section 5 for details): Let $\left( T, V^* \right)$ 
be a pair of twisted operators on a Hilbert space $\mathcal{H}$ such 
that $T$ is a contraction and $V$ is an isometry. Let $S$ on $\mathcal{K}$ 
be the minimal isometric dilation for $T$. If $\widetilde{V}$ on 
$\mathcal{K}$ is an extension of the isometry $V$, then $( S,\widetilde{V} )$  
is a pair of doubly twisted isometries on the minimal space $\clk$. 
Moreover, the pair $(S, \widetilde{V} )$ on the minimal space $\clk$ 
is unique up to unitary equivalence.   
\vspace{.1cm}

The plan for the paper is as follows. In Section 2, we discuss
some basic definitions and the canonical decomposition for a 
contraction as well as classical Wold-von Neumann decomposition 
theorem for isometry. In Section 3, we investigate decomposition for 
various pairs of twisted operators. In Section 4, we study the structure 
of a twisted commutant of a power partial isometry. Section 5 
is devoted to dilation of a contraction and we have shown that the 
Wold-type decomposition for a certain pair of operators holds on the 
minimal isometric dilation space. Finally, we have given a characterization
for doubly twisted isometries in Section 6.

\newsection {Preliminaries}
Through out this paper, $\clh$ denotes a complex separable Hilbert space,  
$\mathcal{B}(\clh)$ as the algebra of all bounded linear operators 
(operators for short) on $\clh$, and $P_{\cll}$ is the orthogonal 
projection of $\clh$ onto $\cll$. For $X \in \mathcal{B}(\clh)$,
$\{ X \}^{'}$ denotes the commutant of $X$. 
A closed subspace $\mathcal{M}$ of $\mathcal{H}$ is 
invariant under $T \in \clb(\clh)$ if $T(\mathcal{M}) \subseteq 
\mathcal{M}$ and subspace $\mathcal{M}$ reduces $T$ if $T(\mathcal{M}) 
\subseteq \mathcal{M}$ and $ T(\mathcal{M}^{\perp}) \subseteq 
\mathcal{M}^{\perp}$. A contraction $T$ on $\clh$ (that is, $\|Th \| 
\leq \|h \|$ for all $h \in \clh$) is said to be a pure contraction if 
$T^{*m} \rightarrow 0$ as $m \rightarrow \infty$ in the strong operator 
topology. A contraction $T$ on $\mathcal{H}$ is called completely non-unitary 
(c.n.u. for short) if there does not exist any nonzero $T$-reducing 
subspace $\mathcal{L}$ of $\mathcal{H}$ such that $T|_{\mathcal{L}}$ is 
a unitary operator. A closed subspace $\mathcal{W} \subseteq \mathcal{H}$ 
is said to be a \textit{wandering subspace} of an isometry $V$ (that is, 
$V^{*}V =I_{\clh}$) if 
\[
V^{k} \mathcal{W} \perp V^{\ell} \mathcal{W}   \ \text{for  all}
 \ \  {k , \ell} \ \in \   \mathbb{Z}_{+} \ \text{with} \ k \neq \ell. 
\]
An isometry $V$ on $\mathcal{H}$ is called a {\textit {unilateral shift 
or shift}} if $ \mathcal{H}= \bigoplus_{m \geq 0} V^m \mathcal{W}$ for 
some wandering subspace $\mathcal{W}$ of $V$. Equivalently, an isometry 
$V$ on $\clh$ is said to be a pure isometry or shift if $V^{*m} \rightarrow 0$ 
as $m \rightarrow \infty$ in the strong operator topology (see Halmos 
\cite{HALMOS-BOOK}). It is noted that if $\clw$ is a \textit{wandering 
subspace} of a shift $V$ on $\clh$, then
\[
\clw = \clh \ominus V \clh
\]
and the dimension of $\clw$ is called the multiplicity of the shift $V$
(see \cite{NF-BOOK}).
\vspace{.2cm}

One of the fundamental results in dilation theory is that every 
contraction on Hilbert spaces can be decomposed into direct sum 
of unitary and c.n.u. We refer this as \textit{canonical decomp
osition theorem} for a contraction (\cite{NF-BOOK}). 

\begin{thm}
Every contraction $T$ on a Hilbert space $\mathcal{H}$ corresponds 
a unique decomposition of $\mathcal{H}$ into an orthogonal sum of 
two $T$-reducing subspaces $\mathcal{H}={\mathcal{H}}_u \oplus 
{\mathcal{H}}_{\neg u}$ such that $T|_{{\mathcal{H}}_u}$ is unitary 
and $T|_{{\mathcal{H}}_{\neg u}}$ is c.n.u. ( ${\mathcal{H}}_u$ or 
${\mathcal{H}}_{\neg u}$ may equal to $\{0\}$).
Moreover
\[
{\mathcal{H}}_u =	\{ h \in \clh : \|T^n h\|= \|h\|= \|{T^*}^n h\| 
~~\mbox{for}~~ n=1,2,, \ldots \}.
\]
Here $T_u=T|_{{\mathcal{H}}_u}$ and $T_{\neg u}=T|_{{\mathcal{H}}_{\neg u}}$ 
are called unitary part and c.n.u. part of $T$, respectively and 
$T=T_u \oplus T_{\neg u}$ is called the canonical decomposition for $T$.
\end{thm}

In particular, \textit{canonical decomposition theorem} for an isometry
coincides with the classical \textit{Wold-von Neumann decomposition}
(\cite{NF-BOOK}).

\begin{thm}\label{thm-Wold}
	Let $V$ be an isometry on a Hilbert space $\clh$ and $\clw = \clh \ominus V \clh$.
	Then $\clh$ decomposes uniquely as a direct sum of two $V$-reducing subspaces
	$\clh_s = \displaystyle{\mo_{m=0}^\infty} V^m \clw$ and
	$\clh_u = \clh \ominus \clh_s$ and
	\[
	V = \begin{bmatrix} V_s & O\\ O & V_u
	\end{bmatrix} \in \clb(\clh_s \oplus \clh_u),
	\]
	where $V_s = V|_{\clh_s}$ is shift and $V_u = V|_{\clh_u}$ 
	is unitary.
\end{thm}

There are some certain classes of contractions which are important 
in understanding the structure of a contraction (see \cite{NF-BOOK}). 
We say that a contraction $T \in \clb(\clh)$ belongs to the class 
$C_{.0}$ if
\[
\lim_{n \rightarrow {\infty}}  \| T^{*n} h \|=0  \mbox{~for all~} h \in \clh,
\]
and $T$ belongs to the class $C_{.1}$ if
\[
\inf_{n}  \| T^{*n} h \|>0  \mbox{~for all non zero~} h \in \clh.
\]
Also a contraction $T$ belongs to the class $C_{0.}$ (or $C_{1.}$ ) 
if $T^*$ belongs to $C_{.0}$ (or $C_{.1}$ ). The class $C_{0.} \cap C_{.0}$ 
is denoted by $C_{00}$, that is, a contraction $T \in \clb(\clh)$ belongs 
to the class $C_{00}$ if
\[
\lim_{n \rightarrow {\infty}}  \| T^nh \|= 0 = \lim_{n \rightarrow {\infty}}  \| T^{*n} h \| 
\mbox{~for all~} h \in \clh.
\]
The class $C_{11}$ is defined as $C_{1.} \cap C_{.1}$.
For example, an operator $T \in \clb(\ell^2(\mathbb{Z}_+))$ is 
defined by
\[
Te_n = \frac{1}{n+2}e_n \quad (n \in \mathbb{Z}_{+}, \{e_n\} 
\mbox{~standard orthonormal basis for} 
~\ell^2(\mathbb{Z}_+))
\]
and belongs to the class $C_{00}$. 

An operator $T$ on $\mathcal{H}$ 
is called a {\it{partial isometry}} if $|| Tx||= ||x||$ for every 
$x \in (\ker T)^\perp$.  The space $(\ker T)^\perp$ is called the initial 
space of $T$, and $\mbox{ran} ~T$ is called its final space. We say that $T$ 
is a power partial isometry if each $T^n$ is a partial isometry for 
$n \geq 1$. It is easy to see that isometry, co-isometry are examples 
of power partial isometry. 
\vspace{.1cm}

Let us turn to the definition of tuples of twisted and  doubly twisted operators.  
The notion of $\clu_n$-twisted contractions was introduced in our earlier paper
\cite{MM-Wold-type}. We shall here call it as doubly twisted contractions.

Let $n>1$ and for $1 \leq i < j \leq n$,   $\{U_{ij}\}$ be $\binom{n}{2}$ 
commuting unitaries on a Hilbert space $\clh$ such that $U_{ji}:=U^*_{ij}$. 
Then we refer $\{U_{ij}\}_{i <j}$ as a twist on $\clh$. 
\begin{defn}(Twisted)
An $n$-tuple $(T_1, \ldots, T_n)$ of operators on $\clh$ is said to be 
a twisted with respect to a twist $\{U_{ij}\}_{i<j}$ if 
\[
T_iT_j=U_{ij}T_jT_i \hspace{1cm}
\mbox{and} \hspace{1cm} T_kU_{ij} =U_{ij}T_k
\]	
for all $i,j,k=1,\dots, n$ and $i \neq j$.
\end{defn}

\begin{defn}(Doubly Twisted)
An $n$-tuple $(T_1, \dots, T_n)$ of operators on $\clh$ is said to be 
doubly twisted with respect to a twist $\{U_{ij}\}_{i<j}$ if 
\[
T_iT_j=U_{ij}T_jT_i; \hspace{0.5cm}
 \hspace{0.5cm} T_i^*T_j= U^*_{ij}T_jT_i^* \hspace{0.5cm}
 \mbox{and} \hspace{1cm} T_kU_{ij} =U_{ij}T_k
\]
for all $i,j,k=1, \dots, n$ and $i \neq j$. 
\end{defn}

In particular, if $U_{ij}=I_{\clh}$ for $1 \leq i < j \leq n$, then $n$-tuple 
$(T_1, \dots, T_n)$ of twisted (or doubly twisted) becomes $n$-tuple of commuting 
(or doubly commuting) operators.
In particular, a pair $(T_1, T_2)$ of operators on $\clh$ is called twisted 
pair with respect to a twist $\{U \}$ if $T_1T_2=UT_2T_1$ and $ T_iU =UT_i$ for 
$i=1, 2$. A pair $(T_1, T_2)$ of operators on $\clh$ is called doubly twisted pair 
with respect to a twist $\{U \}$ if $T_1T_2=UT_2T_1$, $T_1^*T_2= U^*T_2T_1^*$ and
$T_iU =UT_i$ for $i=1, 2$. We shall call $(T_1, T_2)$ as twisted operators or a 
doubly twisted operators without referring a twist $\{U \}$. A commuting pair 
$(T_1, T_2)$ of operators on $\clh$ is said to be doubly commuting if 
$T_1T_2^* = T_2^*T_1$.

\newsection{Decomposition for pairs of operators}

In this section, we studied decomposition results for pairs 
of doubly twisted as well as twisted contractions.The following 
result, which has been shown in our earlier paper \cite{MM-Wold-type} 
will be used frequently in the sequel. However, we present new 
proof for the sake of completeness.

\begin{thm}\label{main thm-reduce}
Let $\left( T,V \right)$ be a pair of doubly twisted operators on 
a Hilbert space $\clh$ such that $T$ is a contraction. Let $\clh 
=\clh_u \oplus \clh_{\neg u}$ be the \textit{canonical decomposition} 
for the contraction $T$. Then the decomposition reduces the operator $V$.   
\end{thm}

\begin{proof}
Suppose that $T \in \clb({\clh})$ is a contraction. Then
$T^*T \leq I_{{\mathcal{H}} }$ and also $T T^* \leq {I_{{\clh}}}.$   
Now define the defect operators
\[
	{D_T}=(I_{{\mathcal{H}}}-T^*T)^{1/2}  \mbox{~~ and~~} 
	\  {D_{T^*}}=(I_{{\mathcal{H}}}-T T^*)^{1/2} 
\]
which are positive operators and bounded by $0$ and $1$. Since $(T,V)$ is 
a pair of doubly twisted operators, there is a unitary operator (twist) 
$U \in \clb(\clh)$ with $ T,V \in \{ U \}'$ such that
\[
TV=UVT \mbox{~~~~and~~~~} VT^*=UT^*V.
\]
Thus
\[
VD_T^2= V-VT^*T=V-UT^*VT=V-UT^*U^*TV=(I_{\clh} - T^*T)V = D_T^2V.
\]
Now using iteration 
\[
V(D_T^2)^n=(D_T^2)^nV  \ \text{for} \ n=0,1,2,\ldots.
\]
Therefore
\begin{align}\label{Thm-equ1}
V p(D_T^2)=p(D_T^2)V 
\end{align}
for every polynomial 
$p(\lambda)=\alpha_0 +\alpha_1 \lambda + \cdots +\alpha_k \lambda^k$. 
Thus we can choose a sequence of polynomials $p_n(\lambda)$ 
that tends to the function $\lambda^{1/2}$ uniformly in the 
interval $0 \leq \lambda \leq 1$. Since $D_T^2$ is a positive 
operator also bounded by $0$ and $1$, by spectral representation 
of $D_T^2$ there is a sequence of operators $p_n(D_T^2)$ 
converges to $D_T$ in the operator norm. Letting $n \rightarrow \infty$ 
from (\ref{Thm-equ1}), we get
\begin{equation*}
VD_T=D_TV.
\end{equation*}
Similarly, we obtain
\[
 VD_{T^*}=D_{T^*} V.
\]

\NI 
Suppose $T \in \clb(\clh)$ and $h \in \clh$ such that $\|Th \| = \|h \|$.
Then
\[
<h, h > = \|h\|^2 = \|Th\|^2 = <T^*Th, h >.
\]
Thus $<D_T^2 h, h> =0$ if and only if $\|D_T h \| = 0$. Hence the set 
$\{ h \in \clh : \|Th\|=\|h\|\}$ coincides with 
$\cln_{D_T}=\{h \in \clh: D_Th=0\}$ (or $\cln_{D_T}$= ker($D_T$)) 
which is a subspace of $\clh$.
Consider $T(m)=T^m \ (m \geq 1),  \ \ T(0)=I, \ \ T(m)=T^{*|m|} \ \ (m \leq -1)$.
For fixed integer $m \in \mathbb{Z}$, $T(m)$ is a contraction on $\clh$. 
Therefore, the set $\{h \in \clh : \|T(m)h\|= \|h\| \}$ is same as 
$\cln_{D_{T(m)}}=\{h \in \clh : {D_{T(m)}}h = 0\}$ which is a  
subspace of $\clh$. Again consider 
$\clh_u=\{h \in \clh: \|T(m)h\|=\|h\| , \ m \in \mathbb{Z}\}$. Then
$\clh_u$ can be expressed as 
\begin{align*}
\clh_u = \bigcap_{m =-\infty}^{\infty} \cln_{D_{T(m)}},
\text{where} \ D_{T(m)}&= 
\begin{cases}
(I-T^{*m}T^m)^{\frac{1}{2}} & \forall \ m \geq 0,\\
(I-T^{|m|}T^{*|m|})^{\frac{1}{2}} & \forall \ m \leq -1.
\end{cases}
\end{align*} 
\NI
Since $D_T$ and $D_{T^*}$ are self-adjoint operators and the pairs 
$(V,D_T)$ and $(V, D_{T^*})$ are commuting, 
the pairs $(V, D_{T^m})$ and $(V, D_{T^{*|m|}})$ are doubly commuting.
Let $h \in \clh_u= \bigcap_{m =-\infty}^{\infty} \cln_{D_{T(m)}}$. 
Then for each fixed $m \geq 0$ 
 \[
  D_{T^m}Vh = V D_{T^m}h =0 \ \text{and} \  D_{T^m}V^*h = V^* D_{T^m}h =0. 
 \]
Therefore $ Vh$ and  $V^*h \in  \cln_{D_{T^m}}$. 
Similarly $ Vh$ and $V^*h \in  \cln_{D_{T^{*|m|}}}$ \text{for} each fixed $ m \leq -1.$ 
It says that $ V(\clh_u), V^*(\clh_u) \subseteq \clh_u $. 
Again $\clh_u$ and $\clh_{\neg u}$ are orthogonal subspaces of $\clh$. 
Hence the canonical decomposition $\clh=\clh_u \oplus \clh_{\neg u}$ reduces $V$.
\end{proof}

The immediate consequences of the above result are as follows:
\begin{rem}
Let $\left( T_1, T_2 \right)$ be a pair of doubly twisted  
contractions on $\clh$. Let $\clh= \clh_{u1}\oplus \clh_{\neg u1}$ 
be the {\textit{canonical decomposition}}
for the contraction $T_1$, where $T_1|_{\clh_{u1}}$ is unitary 
and $T_1|_{\clh_{\neg u1}}$ is c.n.u. Then by the above Theorem 
\ref{main thm-reduce}, the decomposition reduces the other contraction 
$T_2$. Indeed, the contraction $T_2$ breaks into two contractions 
$T_2|_{\clh_{u1}}$ and $T_2|_{\clh_{\neg u1}}$ with respect to 
the decomposition $\clh= \clh_{u1} \oplus \clh_{\neg u 1}$.
Now repeating the process for the contractions $T_2|_{\clh_{u1}}$ 
and $T_2|_{\clh_{\neg u1}}$, we have 
\[
\clh = \clh_{uu} \oplus \clh_{u\neg u} \oplus \clh_{\neg u u} 
\oplus \clh_{\neg u\neg u},
\]
where 
\begin{itemize}
\item $T_1|_{\clh_{uu}}$ and $T_2|_{\clh_{uu}}$ are both unitaries, 

\item $T_1|_{\clh_{u \neg u}}$ is unitary and $T_2|_{\clh_{u \neg u}}$ is c.n.u.,

\item $T_1|_{\clh_{\neg u  u }}$ is c.n.u. and $T_2|_{\clh_{ \neg u  u}}$ is unitary,

\item $T_1|_{\clh_{\neg u \neg u}}$ and $T_2|_{\clh_{\neg u \neg u}}$ are both c.n.u.
\end{itemize}
For more details one can see the recent paper of the authors \cite{MM-Wold-type}.
\end{rem}

In particular, if the twist $U=I_{\clh}$, then the pair $\left( T,V \right)$ 
of doubly twisted operators becomes doubly commuting and we state in the following
(follows from Theorem \ref{main thm-reduce}).

\begin{cor}\label{DC-reduce}
Let $\left( T,V \right)$ be a pair of doubly commuting operators 
on a Hilbert space $\clh$ such that $T$ is a contraction. Then 
the canonical decomposition $\clh=\clh_u \oplus \clh_{\neg u}$  
for $T$ reduces the operator $V$.   
\end{cor}

\begin{rem}
Since $T \in \{ U \}'$ and $U$ is unitary, the pair 
$\left( T, U \right)$ is doubly commuting and hence the 
canonical decomposition $\clh=\clh_u \oplus \clh_{\neg u}$  
for $T$ reduces $U$.  
\end{rem}

The following decomposition result holds for certain pairs 
of twisted operators.

\begin{thm}\label{Decompositon-lemma1}
Let $(T, V)$ be a pair of twisted operators on a Hilbert space 
$\mathcal{H}$ and let $T$ be a contraction. If $\mathcal{H}=\mathcal{H}_{u} 
\oplus \mathcal{H}_{\neg u}$ is the \textit{canonical decomposition} for $T$ 
and $T|_{\mathcal{H}_{\neg u}}$ is in the class $C_{00}$, then the 
decomposition $\mathcal{H}=\mathcal{H}_{u} \oplus \mathcal{H}_{\neg u}$
reduces $V$.
\end{thm}

\begin{proof}
Suppose that $\mathcal{H}=\mathcal{H}_{u} \oplus \mathcal{H}_{\neg u}$
is the canonical decomposition for the contraction $T \in \clb(\clh)$
and $T|_{\mathcal{H}_{\neg u}}$ is in the class $C_{00}$. Then the 
matricial representation for $T$ with the canonical decomposition 
$\mathcal{H}=\mathcal{H}_{u} \oplus \mathcal{H}_{\neg u}$ is
$ \begin{bmatrix}
 T_u & O\\ 
 O & T_{\neg u}
\end{bmatrix},$ 
where $T_u = T|_{\mathcal{H}_{u}}$ and 
$T_{\neg u} = T|_{\mathcal{H}_{\neg u}}$.

Since the pair $(T,V)$ is twisted, there is a unitary (twist) 
$U \in \clb(\clh)$ such that
\[
TU=UT, \quad VU=UV \mbox{~~~~and~~~~} TV=UVT.
\]
Now the pair $(T, U)$ is doubly commuting on $\clh$ as $U$ is unitary. 
Therefore, from Corollary \ref{DC-reduce} the canonical decomposition 
$\mathcal{H}=\mathcal{H}_{u} \oplus \mathcal{H}_{\neg u}$ reduces $U$. 
Hence the matricial representation of $U$ on $\mathcal{H}= \mathcal{H}_{u} 
\oplus \mathcal{H}_{\neg u}$ is 
$\begin{bmatrix}
U_1  & O \\
O    & U_2
\end{bmatrix}$,
where $U_1=U|_{\mathcal{H}_{u}}$ and $U_2= U|_{\mathcal{H}_{\neg u}}$ 
are unitaries on $\clh_u$ and $\clh_{\neg u}$, respectively.   
Suppose that the matricial representation for $V$ is
$\begin{bmatrix} 
A & B\\ 
C & D 
\end{bmatrix}$ 
with respect to the decomposition $\mathcal{H}=\mathcal{H}_{u} 
\oplus \mathcal{H}_{\neg u}$. Then the condition $TV=UVT$ 
implies that
\[ 
\begin{bmatrix} 
T_u A & T_u B\\ 
T_{\neg u} C & T_{\neg u} D 
\end{bmatrix}
= \begin{bmatrix}
	U_1  & O \\
	O    & U_2
\end{bmatrix}
\begin{bmatrix} 
A T_u & B T_{\neg u}
\\ C T_u & D T_{\neg u}
\end{bmatrix}
= \begin{bmatrix} 
	U_1 A T_u & U_1B T_{\neg u}
	\\ U_2C T_u & U_2D T_{\neg u}
\end{bmatrix}.
\]
Therefore,
\[
 T_u B = U_1 B T_{\neg u}  \ \text{and}  \ \ 
 T_{\neg u} C = U_2 C T_u.
\]
Thus
\[
 T_u^n B = U_1^n B T_{\neg u}^n  
\ \ \text{and} \ \
 T_{\neg u}^n C = U_2^n C T_u^n
 \quad \mbox{for all}~ n\geq 1.
\]
Now for any $h \in \clh_{\neg u}$,
\[
\| Bh \| = \| T_u^n B h \| = \| U_1^nB T_{\neg u}^nh \|=\| B T_{\neg u}^nh \| \leq 
\|B \|  \| T_{\neg u}^nh \| \rightarrow 0
\]
as $n \rightarrow \infty$. Therefore, $Bh = 0$ for all 
$h \in \clh_{\neg u}$. This implies that $B=0$. Again for 
any $h \in \clh_{\neg u}$,
\[
\| C^*h \| = \| T_u^{*n} C^* h \|= \|U_2^{n} C^* T_{\neg u}^{*n} h \| 
= \| C^* T_{\neg u}^{*n} h \| 
\leq \|C^* \| \| T_{\neg u}^{*n}h \| \rightarrow 0
\]
as $n \rightarrow \infty$. This implies that $ C^*h =0$ for all 
$h \in \clh_{\neg u}$. Therefore $C=0$. Hence $\mathcal{H}_{u}$ 
reduces $V$. Therefore, the decomposition 
$\mathcal{H}=\mathcal{H}_{u} \oplus \mathcal{H}_{\neg u}$
reduces $V$.

This completes the proof.	
\end{proof}

There are few remarks and consequences of the above theorem:	

\begin{rem}
One of the important facts in the above theorem is that the decomposition 
for the pair of operators does not require the doubly twisted condition. 
\end{rem}

\begin{rem}
If the c.n.u. parts for a pair of twisted contractions are in $C_{00}$, 
then one can find decomposition for the twisted pairs which need not be 
doubly twisted  (see Example \ref{Ex-nondoubly-commuting}).  
\end{rem}

Now we will give some concrete examples of pairs of twisted operators 
satisfying certain properties.

\begin{ex}
Let $\Lambda=\{(i,j):i \geq 0 \ \text{or} \ j \geq 0\}$ and 
${\{e_{i,j}\}}_{(i,j)\in \Lambda}$ be a sequence of orthonormal vectors 
in some Hilbert space $\mathcal{K}$. Let 
$\mathcal{H}= \overline{\text{span}}{\{e_{i,j}\}}_{(i,j) \in \Lambda}$.
Suppose that $T_1,T_2$ are two operators on $\mathcal{H}$ such that 
\[
T_1 e_{i,j}=e_{i+1,j} \quad \mbox{and} \quad T_2 e_{i,j}=e_{i,j+1} \quad 
\mbox{for}~ (i,j) \in \Lambda.
\]
Clearly, $(T_1, T_2)$ is a pair of commuting isometries, that is,
twisted isometries on $\mathcal{H}$ with respect to the twist $I_{\clh}$.
Suppose that $\mathcal{H}={\mathcal{H}}_{u2} \oplus {\mathcal{H}}_{s2} $ 
is the Wold decomposition for $T_2$ on $\mathcal{H}$. Then it is easy to see 
that
\[
\mathcal{H}_{u2}= \bigcap_{n=0}^{\infty}T_2^n \clh = \overline{\text{span}}
\{e_{i,j} : i\geq 0 \}.
\]
Now $e_{0,1} \in \mathcal{H}_{u2}$ but $e_{-1,1} \notin \mathcal{H}_{u2}$. 
Again $T_1^*e_{0,1}=e_{-1,1}$. Therefore, $\mathcal{H}_{u2}$ does not reduce 
the operator $T_1$ and hence, by Theorem \ref{main thm-reduce}, the pair 
$(T_1,T_2)$ is not doubly twisted. Also $T_2|_{{\mathcal{H}}_{s2}}$ does not 
belong to $C_{00}$. Similarly, $T_1|_{{\mathcal{H}}_{s1}}$ does not belong 
to $C_{00}$. Hence the pair $(T_1, T_2)$ of twisted operators is neither 
doubly twisted nor $T_1|_{{\mathcal{H}}_{s1}}$, $T_2|_{{\mathcal{H}}_{s2}}$ 
belong to $C_{00}$. 
\end{ex}

\begin{ex}\label{Counter example-Twisted isometries}
Let $H^2(\mathbb{D})$ denotes the Hardy space over the unit disc $\mathbb{D}$. 
Now for fixed $r$ with $|r|=1$, define a weighted shift operator $S_r$ on 
$H^2(\mathbb{D})$ as
\[
S_r z^n= r^{n+1} z^{n+1}\hspace{1cm}  (n \in \mathbb{Z}_+) 
\]
where $\{1,z,z^2, \dots \}$ is an orthonormal basis for 
$H^2(\mathbb{D})$. Let $M_z$ be the multiplication operator on 
$H^2(\mathbb{D})$ by the coordinate function $z$. Then 
\begin{equation*}
		( M_z S_r)(z ^n)= r^{n+1} z^{n+2} 
		\quad \text{and} \quad
		(  S_r M_z)(z^n)= r^{n+2} z^{n+2}
		\ \ \text{for  $n \in \mathbb{Z}_+$}.
\end{equation*}
Again,
\begin{equation*}
		( M_z^* S_r)(z ^n)=
		r^{n+1} z^{n} \text{ ~~$\forall \ n \geq 0$}\\
\end{equation*}
and 
\begin{equation*}
		(S_r M_z^*)(z ^n)=
		\begin{cases}
			r^{n} z^{n} & \text{if $n \geq 1$}\\
			0 & \text{if $n = 0$}.
		\end{cases}
\end{equation*}
Hence $ S_r M_z=rM_zS_r$ but $ M_z^*S_r \neq  rS_r M_z^*$. Now 
define a unitary $U$ on $H^2(\mathbb{D})$ as $U(z^n)=rz^n$
for $n \geq 0$. Then it is easy to see that the pair $(S_r, M_z)$ is 
a twisted contractions with the twist $U$ on $H^2(\mathbb{D})$, but 
not doubly twisted. Also the c.n.u. parts of both the operators 
are $H^2(\mathbb{D})$, and do not belong to $C_{00}$. 
\end{ex}

\begin{ex}\label{Ex-nondoubly-commuting}		
Let $\clk$ be a Hilbert space and $k$ be a fixed positive integer. 
Define an operator $T$ on the $k$-fold direct sum $\clh={\clk 
\oplus \cdots \oplus \clk}$ as
\[
T(h_1,h_2,...,h_k) = (0,h_1,...,h_{k-1}) \quad \mbox{for}~h_i \in \clk.
\]
Then $(T, T)$ is a twisted pair with a twist $I_{\clh}$ (or commuting pair)
of truncated shifts of index $k$ on $\clh$. Clearly, the adjoint of $T$ 
on $\clh$, denoted as $T^*$, is defined by
\[
T^*(h_1,h_2,...,h_k) = (h_2, \ldots, h_{k}, 0).
\] 
It is easy to see that
\[
TT^* \neq T^*T.
\]
Hence the pair $(T, T)$ is not doubly twisted with a twist $I_{\clh}$ 
(or not doubly commuting pair). Now from the canonical decomposition for 
the contraction $T$, the c.n.u. part is $\clh$ (as unitary part is absent). 
Further, $T$ is in $C_{00}$, in general. 
\end{ex}

\begin{ex}\label{Example-doubly twisted }
The weighted shift $M_z^{\alpha}$ on the Hardy space $H^2(\mathbb{D})$ is 
defined by $M_z^{\alpha}(f)= \alpha zf$ for all $f \in  H^2(\mathbb{D})$, 
where $z$ is the co-ordinate function and $|\alpha| \leq 1$. 
For each fixed $r$ with $|r|=1$, define an operator $A_r$ on $H^2(\mathbb{D})$ 
as
\[
A_r z^n= r^n z^n \hspace{1cm}  (n \in \mathbb{Z}_+), 
\]
where $\{1,z,z^2, \dots \}$ is an orthonormal basis for $H^2(\mathbb{D})$. 
Then 
\begin{equation*}
( M_z^{\alpha} A_r)(z ^n)=  \alpha r^n z^{n+1} \quad \text{and} \quad
( A_r M_z^{\alpha})(z^n) = \alpha r^{n+1} z^{n+1} ~~ \text{for}~~ n \in \mathbb{Z}_+.
\end{equation*}
Again
\begin{equation*}
\big([M_z^{\alpha}]^* A_r\big)(z ^n)=
\begin{cases}
			\bar{\alpha} r^n z^{n-1}, & \text{if $n \geq 1$}\\
			0 & \text{if $n = 0$},
\end{cases}
\end{equation*}
and 
\begin{equation*}
		\big(A_r [M_z^{\alpha}]^*\big)(z ^n)=
		\begin{cases}
			\bar{\alpha} r^{n-1} z^{n-1} & \text{if $n \geq 1$}\\
			0 & \text{if $n = 0$}.
		\end{cases}
\end{equation*}
Hence $ A_r M_z^{\alpha}=rM_z^{\alpha}A_r$ and    
$ [M_z^{\alpha}]^*A_r= rA_r [M_z^{\alpha}]^*$. 
Clearly, $(A_r, M_z^{\alpha})$ is a pair of doubly twisted operators 
with respect to the twist $U=rI$ on $H^2(\mathbb{D})$.
Here c.n.u. part of $A_r, M_z^{\alpha}$ is $\{0\}$ and $H^2(\mathbb{D})$,
respectively. Moreover, $M_z^{\alpha} \notin C_{00}$.
\end{ex}

\begin{ex}
Let $T_1, T_2$ be the contractions on $\ell^2(\mathbb{Z})$ defined as
\[
T_1(e_n)= \frac{r^n}{4} e_{n+1}  \hspace{1cm}\text{and}  \hspace{1cm}
T_2(e_n)= \lambda e_{n+1} \hspace{1cm}   
\]
with $|\lambda|<  1$, $|r|=1$ and $\{e_n\}$ is the standard orthonormal 
basis for $\ell^2(\mathbb{Z})$. Let $U =rI$ be a unitary operator on 
$\ell^2(\mathbb{Z})$. Then
\[
T_1T_2(e_n)= \frac{\lambda r^{n+1}}{4}e_{n+2}, \hspace{0.5cm}  
T_2T_1(e_n)= \frac{\lambda r^{n}}{4}e_{n+2}. 
\]
Again 
\[
T_1T_2^*(e_n)= \frac{\overline{\lambda} r^{n-1}}{4} e_n, \hspace{0.5cm} 
T_2^*T_1(e_n)= \frac{\overline{\lambda} r^{n}}{4} e_n.
\]
Clearly, $ T_1T_2=UT_2T_1$ and $T_1^*T_2=U^*T_2 T_1^*$. Therefore, 
$(T_1,T_2)$ is a pair of doubly twisted contractions on $\ell^2(\mathbb{Z})$
with a twist $U$. Moreover, the c.n.u. part for $T_1, T_2$ is $\ell^2(\mathbb{Z})$
and both are in $C_{00}$.
\end{ex}

The following particular result follows from the above Theorem
\ref{Decompositon-lemma1}.

\begin{cor}\label{cor2}
Let $\left( T,V \right)$ be a pair of twisted operators 
on a Hilbert space $\clh$ such that $T$ is a contraction.  
Let $\clh= \clh_u \oplus \clh_{\neg u}$ be the canonical 
decomposition for $T$. If $T|_{\clh_{\neg u}} \in C_{0.}$
(or $T|_{\clh_{\neg u}} \in C_{.0}$), then $\clh_{\neg u}$ 
(or $\clh_{u}$) is invariant under $V$.   
\end{cor}

Now suppose that $\left( T, V \right)$ is a pair of twisted operators 
on a Hilbert space $\clh$ such that $V$ is isometry. Then from the Wold
decomposition for $V$, we have $\clh= \clh_u \oplus \clh_{s}$, 
where $V|_{\clh_u}$ is unitary and $V|_{\clh_s}$ is shift,
that is, $V|_{\clh_s}$ is in $C_{.0}$. If $V$ is a co-isometry, then
we have $\clh= \clh_u \oplus \clh_{s}$, where $V|_{\clh_u}$ is unitary 
and $V|_{\clh_{s}}$ is co-shift, i.e., $V|_{\clh_s}$ is in $C_{0.}$. 
Therefore, we have the following results from the above Corollary.

\begin{lem}\label{lem1}
Let $\left( T,V \right)$ be a pair of twisted  operators 
on a Hilbert space $\clh$ such that $V$ is isometry (or $V$ is co-isometry).  
Let $\mathcal{H}=\mathcal{H}_{u} \oplus \mathcal{H}_{s}$ 
be the Wold decomposition for $V$. Then $\clh_{u}$ (or $\clh_{s}$)
is invariant under $T$. 
\end{lem}

We now present the the well-known result of S\l{}o\'{c}inski 
\cite{SLOCINSKI-WOLD} for a pair of doubly twisted isometries by 
using the above results. One can see our recent paper \cite{MM-Wold-type}
for more results.

\begin{thm}\label{Slocinski result-twist}
Let $\left( V_1, V_2 \right)$ be a pair of doubly twisted isometries 
on a Hilbert space $\mathcal{H}$. Then there is a unique decomposition 
\[
\mathcal{H}=\mathcal{H}_{uu} \oplus \mathcal{H}_{us} \oplus 
\mathcal{H}_{su} \oplus \mathcal{H}_{ss},
\]
where $\mathcal{H}_{uu},  \mathcal{H}_{us}, \mathcal{H}_{su},$ and 
$\mathcal{H}_{ss}$ are the subspaces reducing $V_1$ and $V_2$ such that
\begin{itemize}
\item	$V_1|_{\mathcal{H}_{uu}}, V_2|_{\mathcal{H}_{uu}}$ are unitary operators,
		
\item   $V_1|_{\mathcal{H}_{us}}$ is unitary, $V_2|_{\mathcal{H}_{us}}$ is shift,
		
\item  $V_1|_{\mathcal{H}_{su}}$ is shift, $V_2|_{\mathcal{H}_{su}}$ is unitary, 
		
\item $V_1|_{\mathcal{H}_{ss}}$,$V_2|_{\mathcal{H}_{ss}}$ are shifts.
\end{itemize} 	
\end{thm}

\begin{proof}
Suppose $\left( V_1, V_2 \right)$ is a pair of doubly twisted 
isometries with respect to a twist $U$ on $\mathcal{H}$. Then 
the Wold decomposition for $V_1$ gives
\[
\clh=\clh_u \oplus \clh_s,
\]
where $\clh_u$ reduces $V_1$; $V_1|_{\clh_u}$ is unitary and 
$V_1|_{\clh_s}$ is shift. Since $(V_1,V_2)$ is a doubly twisted 
pair on $\clh$, by Theorem \ref{main thm-reduce} the subspaces $\clh_u$ 
and $\clh_s$ reduce the isometry $V_2$ and the unitary $U$. 
Consequently, $U|_{\clh_u}$ and $U|_{\clh_s}$ are unitary operators
on $\clh_u$ and $\clh_s$, respectively.
Now Wold-von 
Neumann decomposition for the isometries $V_2|_{\clh_u}$ on $\clh_u$ 
and $V_{2}|_{\clh_s}$ on $\clh_s$ yield
\[
\clh_u=\clh_{uu} \oplus \clh_{us}, \quad \mbox{and} \quad \clh_s
=\clh_{su} \oplus \clh_{ss},
\]
where $\clh_{uu}$ and $\clh_{su}$ reduce $V_2$ to unitary operators 
and $\clh_{us},\clh_{ss}$ reduce $V_2$ to unilateral shifts. 
Again the pairs $(V_{1}|_{\clh_u}, V_{2}|_{\clh_u})$ and 
$(V_{1}|_{\clh_s}, V_{2}|_{\clh_s})$ are doubly twisted with respect 
to twist $U|_{\clh_u}$ and $U|_{\clh_s}$ on $\clh_u$ and $\clh_s$, 
respectively. Therefore, by Theorem \ref{main thm-reduce} the subspaces 
$\clh_{uu}$ and $\clh_{us}$ reduce the unitary $V_{1}|_{\clh_u}$ 
as well as subspaces $\clh_{su}$ and $\clh_{ss}$ reduce the shift $V_1|_{\clh_s}$.

This completes the proof.
\end{proof}

Every completely non-unitary co-isometry is a co-shift (that is, 
adjoint of shift) and hence it is in $C_{0.}$. If $T \in \clb(\clh)$ 
is co-isometry, then the canonical decomposition $\clh= \clh_0 \oplus 
\clh_1$ reduce $T$, where $T|_{\clh_0}$ is unitary and $T|_{\clh_1}$ 
is in $C_{0.}$. So in particular, if $(V, W)$ is a twisted pair
consisting of an isometry and a co-isometry, then it is doubly twisted
(For proof see Section 5) and hence we have the following decomposition
result. The proof is omitted as it is similar to the above proof.

\begin{cor}
Let $\left( V, W \right)$ be a pair of twisted operators on a 
Hilbert space $\clh$ such that $V$ is isometry and $W$ is co-isometry. 
Then there is a unique decomposition 
\[
\mathcal{H}=\mathcal{H}_{uu} \oplus \mathcal{H}_{ub} \oplus 
\mathcal{H}_{su} \oplus \mathcal{H}_{sb}
\]
where $\mathcal{H}_{uu},  \mathcal{H}_{ub}, \mathcal{H}_{su}$, and 
$\mathcal{H}_{sb}$ are the subspaces reducing $V, W$ such that
\begin{itemize}
\item	$V|_{\mathcal{H}_{uu}}$ and $W|_{\mathcal{H}_{uu}}$ are unitary operators,
		
\item   $V|_{\mathcal{H}_{ub}}$ is unitary and $W|_{\mathcal{H}_{ub}}$ is co-shift,
		
\item  $V|_{\mathcal{H}_{su}}$ is shift and $W|_{\mathcal{H}_{su}}$ is unitary, 
		
\item $V|_{\mathcal{H}_{sb}}$ is shift and $W|_{\mathcal{H}_{sb}}$ is co-shift.
\end{itemize} 	
\end{cor}

In the next result, we obtain an explicit Wold-type decomposition for a 
pair of twisted contractions satisfying certain conditions.

\begin{thm}
Let $\left( T_1, T_2 \right)$ be a pair of twisted contractions
on a Hilbert space $\clh$. Let $\clh=\clh_{ui} \oplus \clh_{\neg u i}$ 
be the canonical decomposition for $T_i$ for $i=1, 2$. If 
$T_i|_{\clh_{\neg u i}}$ is in $C_{00}$ for $i=1, 2$, then there is a 
unique decomposition 
\[
\mathcal{H}=\mathcal{H}_{uu} \oplus \mathcal{H}_{u \neg u} \oplus 
\mathcal{H}_{ \neg u u} \oplus \mathcal{H}_{ \neg u \neg u},
\]
where $\mathcal{H}_{uu},  \mathcal{H}_{u \neg u}, \mathcal{H}_{ \neg u u},$ and 
$ \mathcal{H}_{ \neg u \neg u}$ are $(T_1, T_2)$ reducing subspaces of $\clh$  
such that
\begin{itemize}
\item	$ T_1|_{\mathcal{H}_{uu}}, T_2|_{\mathcal{H}_{uu}}$ are unitary operators,
		 		
\item   $ T_1|_{\mathcal{H}_{u \neg u}}$ is unitary, $T_2|_{\mathcal{H}_{u \neg u}}$ 
is completely non-unitary,
		 		
\item  $ T_1|_{\mathcal{H}_{ \neg u u}}$ is completely non-unitary, 
$T_2|_{\mathcal{H}_{ \neg u u}}$ is unitary, 
		 		
\item $ T_1|_{\mathcal{H}_{ \neg u \neg u}}$, 
$T_2|_{\mathcal{H}_{ \neg u \neg u}}$ are completely non-unitary operators.
\end{itemize} 	
Furthermore, the decomposition spaces are as follows
\begin{align*}
&\clh_{uu}=
\bigcap_{ m_2 \in \mathbb{Z}_{+}} [\ker ((I -T_2^{* m_2} T_2^{ m_2})|_{\clh_{u1}}) 
\cap \ker ((I -T_2^{ m_2} T_2^{* m_2})|_{\clh_{u1}})],\\
& \clh_{u\neg u}= \bigvee_{  m_2 \in \mathbb{Z}_{+}} \{ (I-T_2^{* m_2}T_2^{ m_2}) \clh_{u1} \cup 
(I-T_2^{ m_2}T_2^{* m_2}){\clh_{u1}} \}, \\
&\clh_{\neg u u}=
\bigcap_{ m_2 \in \mathbb{Z}_{+}} [\ker ( (I-T_2^{*m_2} T_2^{ m_2})|_{\clh_{\neg u 1}})
\cap \ker ((I -T_2^{ m_2} T_2^{* m_2}) |_{\clh_{\neg u 1}})], \\
& \clh_{\neg u\neg u}= \bigvee_{ m_2 \in \mathbb{Z}_{+}} \{ (I-T_2^{* m_2}T_2^{ m_2}) \clh_{\neg u1} \cup 
		(I-T_2^{ m_2}T_2^{* m_2}){\clh_{\neg u1}} \}, 
\end{align*}
and
\begin{align*}
&\clh_{u1}= \bigcap_{ m_1 \in \mathbb{Z}_{+}}[\ker(I- T_1^{* m_1} T_1^{ m_1}) 
		\cap \ker (I- T_1^{ m_1} T_1^{* m_1} )], \\
& \clh_{\neg u1}= \bigvee_{  m_1 \in \mathbb{Z}_{+}} \{ (I-T_1^{* m_1}T_1^{ m_1})\clh
		\cup (I-T_1^{ m_1} T_1^{* m_1})\clh \}.
\end{align*}	
\end{thm}

\begin{proof}
Suppose that $(T_1, T_2)$ is a pair of twisted contractions with respect 
to a twist $U$ on a Hilbert space $\clh$. Then the canonical decomposition 
for $T_1$ gives
\[
\clh=\clh_{u1} \oplus \clh_{\neg u 1} 
\]
where $T_1|_{\clh_{u1}}, T_1|_{\clh_{\neg u 1}}$ are unitary and c.n.u., 
respectively. Also, for the contraction $T_1$ on $\clh$, we have the 
decomposition spaces $\clh_{u1}$ and $\clh_{\neg u 1}$ (see our recent 
work \cite{MM-Wold-type} for more details) as
\begin{align*}
\clh_{u1} 
& = \bigcap_{m_1 \in \mathbb{Z}_{+}}[ \ker (I-T_1^{*m_1}T_1^{m_1}) 
\cap \ker (I-T_1^{m_1}T_1^{*m_1})], \\
\clh_{\neg u1} 
& = \bigvee_ {m_1 \in \mathbb{Z}_{+}} \{ (I-T_1^{*m_1}T_1^{m_1})\clh 
\cup (I-T_1^{m_1} T_1^{*m_1})\clh \}.
\end{align*}
Since by hypothesis $T_1|_{\clh_{\neg u1}} \in C_{00}$, by Theorem 
\ref{Decompositon-lemma1} the subspaces $\clh_{u1}$, $\clh_{\neg u 1}$ 
reduce the contraction $T_2$ as well as the unitary $U$. Thus, 
$U|_{\clh_{u1}}$ and $U|_{\clh_{\neg u1}}$ are unitaries on $\clh_{u1}$ 
and $\clh_{\neg u1}$, respectively. Therefore, $(T_1|_{\clh_{u1}}, 
T_2|_{\clh_{u1}})$ and $(T_1|_{\clh_{\neg u1}}, T_2|_{\clh_{\neg u1}})$ 
are pairs of twisted contractions with respect to the twists $U|_{\clh_{u1}}$ 
and $U|_{\clh_{\neg u1}}$ on $\clh_{u1}$ and $\clh_{\neg u1}$, respectively.
Again the canonical decomposition for the contractions $T_2|_{\clh_{u1}}$ 
and $T_2|_{\clh_{\neg u1}}$ yield 
\[
\clh_{u1}=\clh_{uu} \oplus \clh_{u\neg u}, \ \
\text{and} \ \
\clh_{\neg u 1}=\clh_{\neg u u} \oplus \clh_{\neg u \neg u},
\]
where $T_2|_{\clh_{uu}}$, $T_2|_{\clh_{\neg u u}}$ are unitaries
and $T_2|_{\clh_{u\neg u}}$, $ T_2|_{\clh_{\neg u\neg u}}$ are c.n.u.
By the assumption $T_2|_{\clh_{\neg u2}} \in C_{00}$, that means, 
$T_2|_{\clh_{u\neg u}},T_2|_{\clh_{\neg u\neg u}} \in C_{00}$ as 
$\clh_{\neg u2} = {\clh_{u\neg u}} \oplus \clh_{\neg u\neg u}$. 
Therefore, by Theorem \ref{Decompositon-lemma1}, the subspaces $\clh_{uu}$
and $\clh_{u\neg u}$ reduces the unitary operator $T_1|_{\clh_{u1}}$ and
the subspaces $\clh_{\neg u u}$ and $\clh_{\neg u \neg u}$ reduces the 
completely non-unitary $T_1|_{\clh_{\neg u 1}}$. Now using the decomposition 
formula for the canonical decomposition of the contractions $T_2|_{\clh_{u1}}$ 
and $T_2|_{\clh_{\neg u1}}$, we have
\begin{align*}
\clh_{uu}
& = \bigcap_{ m_2 \in \mathbb{Z}_{+}} [\ker (I_{\clh_{u1}}-T_2^{* m_2}|_{\clh_{u1}}  
T_2^{ m_2}|_{\clh_{u1}})\cap \ker (I_{\clh_{u1}}-T_2^{ m_2}|_{\clh_{u1}} T_2^{* m_2}|_{\clh_{u1}})] \\
& = \bigcap_{m_2 \in \mathbb{Z}_{+}} [\ker ((I -T_2^{* m_2} T_2^{ m_2})|_{\clh_{u1}}) 
\cap \ker (I -T_2^{ m_2} T_2^{* m_2})|_{\clh_{u1}})],\\
\clh_{u\neg u} & = \bigvee_{ m_2 \in \mathbb{Z}_{+}} \{ (I-T_2^{* m_2}T_2^{ m_2}) 
\clh_{u1} \cup (I-T_2^{ m_2}T_2^{* m_2}){\clh_{u1}} \},\\
\clh_{\neg u u} 
& = \bigcap_{ m_2 \in \mathbb{Z}_{+}} [\ker (I_{\clh_{\neg u1}}-T_2^{ *m_2} |_{\clh_{\neg u1 }} 
T_2^{ m_2} |_{\clh_{\neg u1 }}) \cap \ker (I_{\clh_{\neg u 1}}-T_2^{ m_2} |_{\clh_{\neg u1}}
T_2^{* m_2} |_{\clh_{\neg u 1}})] \\
& = \bigcap_{ m_2 \in \mathbb{Z}_{+}} [\ker ( (I-T_2^{*m_2} T_2^{ m_2})|_{\clh_{\neg u 1}})
\cap \ker ((I -T_2^{ m_2} T_2^{* m_2}) |_{\clh_{\neg u 1}})],\\ 
\mbox{and }
\clh_{\neg u\neg u} & = \bigvee_{ m_2 \in \mathbb{Z}_{+}} \{ (I-T_2^{* m_2}T_2^{ m_2}) 
\clh_{\neg u1} \cup (I-T_2^{ m_2}T_2^{* m_2}){\clh_{\neg u1}} \}.
\end{align*}
	
The uniqueness of the decomposition follows from the uniqueness 
of the canonical decomposition of a single contraction. 
This finishes the proof of the theorem.	
\end{proof}

\section{Decomposition for Partial Isometry}

Let $\clh=\underbrace{\clh' \oplus \clh' \oplus \cdots \oplus \clh'}_k$, 
where $\clh'$ is a Hilbert space. A truncated shift $R'$ of index $k$ is 
defined on the $k$-fold direct sum $\clh$ as
\[
R'(h_1,h_2,...,h_k) = (0,h_1,...,h_{k-1}) \quad \mbox{for~} k \in \mathbb{Z}_+.
\]
Therefore, the matrix representation of the truncated shift operator $R'$ 
of index $k$ is of the form 
\[
R'=\begin{bmatrix}
0          &0 &0 &\dots &0 &0 \\
I_{\clh'} &0 &0 &\dots &0 &0 \\
0 &I_{\clh'} &0 &\dots &0 &0\\
\vdots &\vdots &\vdots &\dots &0 &0  \\
0 &0 &0 &\dots &I_{\clh'} &0
\end{bmatrix}_{k \times k}.   
\]
Clearly, it is easy to see that $R' \in C_{00}$ and $\ker R'= \clh' $, 
$R'^*R'=P_{(\ker R')^{\perp}}$. For an example, we consider an operator 
$R $ on $\mathbb{C} \oplus \mathbb{C} \oplus \mathbb{C} \oplus \mathbb{C}$ 
as
\[
	R=\begin{bmatrix}
	0 &0 &1 &0 \\
	a &0 &0 &\sqrt{1-|a|^2}\\
	0 &0 &0  &0 \\
    0 &0 &0  &0
	\end{bmatrix} ,
	\ \ \text{where}\ 0 <|a| <1.
\]
Then one can check that $R$, $R^n(n \geq3)$ are partial isometries but 
$R^2$ is not a partial isometry. Hence $R$ is not a power partial isometry.

Now we recall the decomposition theorem of Halmos and Wallen \cite{HW-POWERS} 
on power partial isometries.

\begin{thm}\label{Power-partial-isometry}
Let $R \in \mathcal{B}(\mathcal{H})$ be a power partial isometry. 
Then there is a unique decomposition  
\[
\displaystyle{ \mathcal{H}=\mathcal{H}_u \oplus \mathcal{H}_s \oplus 
\mathcal{H}_b \oplus (\oplus_{k \geq 1}\mathcal{H}_k}),  
\]
where $\mathcal{H}_u, \mathcal{H}_s, \mathcal{H}_b$ and $\mathcal{H}_k, 
k \geq 1$ are subspaces of $\clh$ reducing $R$ such that
\begin{itemize}
\item $R_u= R|_{\mathcal{H}_u}$ is a unitary operator, 
\item $R_s= R|_{\mathcal{H}_s}$ is a unilateral shift,
\item $R_b= R|_{\mathcal{H}_b}$ is a backward shift, 
\item $R_t= R|_{\mathcal{H}_k}$ is a truncated shift of index $k$. 
\end{itemize} 
Moreover,
\begin{itemize}
\item $\displaystyle{{\mathcal{H}_u} 
=\cap_{n\geq 0} R^{*n}{\mathcal{H}}} \bigcap \cap_{n\geq 0} R^{n}{\mathcal{H}},$ 
\item $\displaystyle{{\mathcal{H}_s}
= \cap_{n\geq 0} R^{*n}{\mathcal{H}} \bigcap \oplus_{n\geq 0} R^{n}(\ker R^*)},$ 
\item  $\displaystyle{{\mathcal{H}_b}
= \cap_{n\geq 0} R^{n}{\mathcal{H}} \bigcap \oplus_{n\geq 0} R^{*n}(\ker R)} ,$ 
\item  $ \clh_t =\oplus_{k \geq 1} \mathcal{H}_k
= \oplus_{n \geq 0} R^{n}(\ker R^*) \bigcap \oplus_{n \geq 0}R^{*n}(\ker R)$.
\end{itemize} 
\end{thm}

Following Halmos and Wallen \cite{HW-POWERS}, one can conclude that 
for a power partial isometry $R$, $R_u= R|_{ \mathcal{H}_u} \in 
C_{11}$, $R_s= R|_{ \mathcal{H}_s} 
\in  C_{10}$, $R_b = R|_{ \mathcal{H}_b}  \in C_{01}$, and $R_k= 
R|_{ \mathcal{H}_k} \in  C_{00}$. Thus any power partial isometry
$R$ on $\mathcal{H}=\mathcal{H}_u \oplus \mathcal{H}_b \oplus \mathcal{H}_s 
\oplus \mathcal{H}_t $ is of following type: 
\[
\bordermatrix{&\clh_u &\clh_b &\clh_s &\clh_t \cr
	& C_{11} &O &O &O \cr
	&O & C_{01} &O &O \cr
	&O &O  & C_{10} &O\cr
    &O &O &O  & C_{00} \cr}.
\]

Firstly, we find matricial representation for a doubly commutant of 
a power partial isometry. 
\begin{prop}\label{DC-power partial isometry}
Let $(R,B)$ be a pair of doubly commuting operators on $\mathcal{H}$ 
such that $R$ is a power partial isometry. If 
$\mathcal{H}= \mathcal{H}_{u} \oplus \mathcal{H}_b \oplus 
\mathcal{H}_s \oplus \mathcal{H}_t$ 
is the decomposition for $R$, then this decomposition reduces also 
the operator $B$.
\end{prop}

\begin{proof}
Suppose that $(R,B)$ is a pair of operators such that $R$ 
is a power partial isometry and 
$\mathcal{H}= \mathcal{H}_{u} \oplus \mathcal{H}_b \oplus 
\mathcal{H}_s \oplus \mathcal{H}_t$ is the decomposition for $R$. 

Let $\clk_1=\clh_u \oplus \clh_b$ and $\clk_2=\clh_s \oplus \clh_t$. 
If $RB=BR$, then one can prove (similar to proof of the Theorem 
\ref{Decompositon-lemma1}), $\clk_1$ is an invariant subspace for $B$. 
So the matrix representation of $B$ with respect to the decomposition 
$\clh=\clk_1 \oplus \clk_2$ is of the form 
\[
	 	 B = \begin{bmatrix} 
	 	 B_1 & *\\ 
	 	 O & B_2
	 	\end{bmatrix},
	   \	\text{where} \ \ B_1=B|_{\clk_1},
	 	\	\text{and} \ B_2= P_{\clk_2}B|_{\clk_2}.
\]
Now $B_1$ is a bounded operator on $\clk_1=\clh_u \oplus \clh_b$
and let $B_1= 
\begin{bmatrix} 
B_{11} & B_{21}\\ 
B_{31} & B_{41} 
\end{bmatrix}$
on $\clk_1$. Then
\[
R|_{\clk_1} B|_{\clk_1} = B|_{\clk_1} R|_{\clk_1} 
\]
yields
\[
\begin{bmatrix} 
R_u B_{11} & R_u B_{21}\\ 
R_b B_{31} & R_b B_{41} 
\end{bmatrix}
= 
\begin{bmatrix} 
B_{11} R_u & B_{21} R_b\\ 
B_{31} R_u & B_{41} R_b 
\end{bmatrix}.
\] 
Therefore, $ R_u B_{21}= B_{21} R_b$. Since $R_u$ is unitary 
and $R_b$ is in $C_{0.}$, by Corollary \ref{cor2} we have $B_{21} =0$. 
Therefore, the matrix form of $B_1$ on $\clk_1=\clh_u \oplus \clh_b$ 
is of the form 
\[
\begin{bmatrix} 
B_{11} & O\\ 
B_{31} & B_{41} 
\end{bmatrix}.
\] 
With the similar arguments, we can find the 
matrix form of $B_2$ on $\clk_2=\clh_s \oplus \clh_t$. Let
\[  
B_2= \begin{bmatrix} 
B_{12} & B_{22}\\
B_{32} & B_{42} 
\end{bmatrix}
\]
on $\clk_2=\clh_s \oplus \clh_t$. Since $R_s$ is an isometry and
$R_t$ is in $C_{00}$, we have $B_2$ is of the form 
\[
\begin{bmatrix} 
B_{33} & O\\
B_{43} & B_{44} 
\end{bmatrix}.
\] 
Hence the matrix form of the operator $B$, commutant of a power partial 
isometry on $\mathcal{H}=\mathcal{H}_u \oplus \mathcal{H}_b \oplus 
\mathcal{H}_s \oplus \mathcal{H}_t $ is of the form 
\[
B=
\bordermatrix{&\clh_u &\clh_b &\clh_s &\clh_t \cr
	&B_{11} &O &* &* \cr
	&* &B_{22}  &* &* \cr
	&O &O  &{B_{33}} &O\cr
	&O &O &*  &{B_{44}} \cr}.
\]
Now if $(R, B)$ is doubly commuting, i.e, $RB=BR$ and $R^*B=BR^*$, 
then $B$ is of the form
\[
\bordermatrix{&\clh_u &\clh_b &\clh_s &\clh_t \cr
	&B_{11} &O &O &O\cr
	&O &B_{22}  &O &O \cr
	&O &O  &{B_{33}} &O\cr
	&O &O &O  &{B_{44}} \cr}.
\]
This shows that the decomposition for a power partial isometry
reduces $B$.
\end{proof}

Now we are in a position to state the main result of this section.

\begin{thm}
Let $(R,Q)$ be a pair of doubly twisted operators on $\mathcal{H}$ 
such that $R$ is a power partial isometry. If 
$\mathcal{H}= \mathcal{H}_{u} \oplus \mathcal{H}_b \oplus \mathcal{H}_s 
\oplus \mathcal{H}_t$ is the decomposition for $R$, then this 
decomposition reduces also the operator $Q$.
\end{thm}

\begin{proof}
Suppose $(R,Q)$ is a pair of doubly twisted operators with respect 
to a twist $U$ on $\mathcal{H}$. Let $R$ be a power partial isometry 
and $\mathcal{H}= \mathcal{H}_{u} \oplus \mathcal{H}_b \oplus 
\mathcal{H}_s \oplus \mathcal{H}_t$ be the decomposition for $R$.

Consider $\clk_1=\clh_u \oplus \clh_b$ and $\clk_2=\clh_s \oplus \clh_t$. 
Since $RQ=UQR$, from the proof of Theorem \ref{Decompositon-lemma1}, 
$\clk_1$ is an invariant subspace for $Q$. So the matrix representation 
of $Q$ with respect to the decomposition $\clh=\clk_1 \oplus \clk_2$ is 
of the form 
\[
Q = \begin{bmatrix} 
		Q_1 & *\\ 
		O & Q_2
	\end{bmatrix},
	\	\text{where} \ \ Q_1=Q|_{\clk_1},
	\	\text{and} \ Q_2= P_{\clk_2}Q|_{\clk_2}.
\]
Now the pair $(R, U)$ is doubly commuting on $\clh$, and hence from 
Proposition \ref{DC-power partial isometry}, the structure of the 
unitary $U$ is of the form	 
\[
	\bordermatrix{&\clh_u &\clh_b &\clh_s &\clh_t \cr
		&U_{11} &O &O &O\cr
		&O &U_{22}  &O &O \cr
		&O &O  &{U_{33}} &O\cr
		&O &O &O  &{U_{44}} \cr},
	\quad \mbox{or} \quad
	\bordermatrix
	{&\clk_1 &\clk_2 \cr
	&U_1 &O\cr
	&O   &U_2},
\]
where
$U_1=
\begin{bmatrix} 
		U_{11} & O\\ 
		O & U_{22} 
\end{bmatrix}$
on $\clk_1$ and 
$U_2
=\begin{bmatrix} 
		U_{33} & O\\ 
		O & U_{44} 
	\end{bmatrix}$
on $\clk_2$. 	
Then
\[
R|_{\clk_i} Q|_{\clk_i} = U_iQ|_{\clk_i} R|_{\clk_i} \quad (i=1, 2). 
\]
Let 
$Q_i= 
	\begin{bmatrix} 
		Q_{1i} & Q_{2i}\\ 
		Q_{3i} & Q_{4i} 
	\end{bmatrix}$ 
on $\clk_i$ for $i=1, 2$. 
Now
\[
R|_{\clk_1} Q|_{\clk_1} = U_1Q|_{\clk_1} R|_{\clk_1}  
\] 
on $\clk_1=\clh_u \oplus \clh_b$ yields
\[
	\begin{bmatrix} 
		R_u Q_{11} & R_u Q_{21}\\ 
		R_b Q_{31} & R_b Q_{41} 
	\end{bmatrix}
	= 
	\begin{bmatrix} 
		U_{11}Q_{11} R_u & U_{11}Q_{21} R_b\\ 
		U_{22}Q_{31} R_u & U_{22}Q_{41} R_b 
	\end{bmatrix}.
\] 
Therefore, $ R_u Q_{21}= U_{11}Q_{21} R_b$. Consequently, 
$(R|_{\clk_1},Q_1)$ is a pair of twisted contractions with 
respect the twist $U_1$ on $\clk_1$. Since $R_u$ is unitary 
and $R_b$ is in $C_{0.}$, by Corollary \ref{cor2} we have 
$Q_{21} =0$. Therefore, the matrix form of $Q_1$ on $\clk_1
=\clh_u \oplus \clh_b$ is of the form 
\[
	\begin{bmatrix} 
		Q_{11} & O\\ 
		Q_{31} & Q_{41} 
	\end{bmatrix}.
\] 
Again $R_s$ is an isometry and $R_t$ is in $C_{00}$ Theorem 
\ref{Decompositon-lemma1}, we have $Q_2$ on $\clk_2=\clh_s \oplus 
\clh_t$ is of the form 
\[
	\begin{bmatrix} 
		Q_{33} & O\\
		Q_{43} & Q_{44} 
	\end{bmatrix}.
\] 
Hence the matrix form  of the twisted commutant $Q$ with
a power partial isometry on 
$\mathcal{H}=\mathcal{H}_u \oplus \mathcal{H}_b \oplus \mathcal{H}_s 
\oplus \mathcal{H}_t $ is of the form 
\[
	Q=
	\bordermatrix{&\clh_u &\clh_b &\clh_s &\clh_t \cr
		&Q_{11} &O &* &* \cr
		&* &Q_{22}  &* &* \cr
		&O &O  &{Q_{33}} &O\cr
		&O &O &*  &{Q_{44}} \cr}.
\]
Since $(R, Q)$ is a doubly twisted pair on $\clh$. Hence  $Q$ is of the form
\[
	\bordermatrix{&\clh_u &\clh_b &\clh_s &\clh_t \cr
		&Q_{11} &O &O &O\cr
		&O &Q_{22}  &O &O \cr
		&O &O  &{Q_{33}} &O\cr
		&O &O &O  &{Q_{44}} \cr}.
\]
This shows that the decomposition for a power partial isometry reduces $Q$.
\end{proof}

\begin{rem}
In the above proof we use the structure theory of a
power partial isometry and hence the proof does not require 
the explicit orthogonal decomposition spaces for power partial isometry. 
However, one can easily prove the above result with the help of the explicit 
decomposition spaces.
\end{rem}

\section{Dilation and wold type decomposition}

In general, the structure of a contraction on a Hilbert space 
is very difficult to study. Using dialtion theory, we show here that
certain pairs of contractions are doubly twisted on the minimal isometric 
dilation space.

\begin{defn}
Let $\mathcal{H} \subset \mathcal{K}$ be two Hilbert spaces. 
Suppose $T \in \mathcal{B}(\mathcal{H})$  and $S \in \mathcal{B}
(\mathcal{K})$ are two bounded operators.
Then $S$ is called a dilation of $T$ if
\[ 
T^n h =P_{\mathcal{H}}S^n h  
\] 
for all $ h \in \mathcal{H}$ and $n \in \mathbb{Z}_+$, 
where $P_{\mathcal{H}}$ is the orthogonal projection of 
$\mathcal{K}$ onto $\mathcal{H}$. A dilation $S$ of $T$ is 
called minimal if 
\[
\overline{span} \left\lbrace S^nh:h \in \mathcal{H},
	 n \in \mathbb{Z}_+   \right\rbrace  =\mathcal{K}. 
\]
An isometric dilation of $T$ is a dilation $S$ which is an isometry. 
\end{defn}

Now we recall one of the striking results on dilation theory (cf. \cite{NF-BOOK}, 
\cite{GPSS-Dilations}).

\begin{thm}
For every contraction $T$ on a Hilbert space $\mathcal{H}$
there exists an isometric dilation $S$ on some Hilbert space 
$\mathcal{K}\left( \supset \mathcal{H} \right) $, which is 
moreover minimal in the sense that 
\[
	  \mathcal{K}= \bigvee_{n=0}^{\infty} S^n \mathcal{H},
	 \ \text{i.e.,} \ 
	  \mathcal{K}= \overline{span} \left\lbrace S^nh:h \in \mathcal{H},
	  n \in \mathbb{Z}_+ \right\rbrace.
\] 
This minimal isometric dilation $ \left( S, \clk \right)$ of
$\left( T, \clh \right)$ is determined upto isomorphism. 
The space $\mathcal{H}$ is invariant for $S^*$ and we have  
\[
 TP_{\mathcal{H}}=P_{\mathcal{H}}S \ \ \text{and} \ \ T^*= S^*|_{\mathcal{H}}, 
\] 
where $P_{\mathcal{H}}$ denotes the orthogonal projection from $\mathcal{K}$ 
onto $\mathcal{H}$.
\end{thm}

One of the main results of this section is as follows:

\begin{thm}\label{Main result}
Let $\left( T, V \right)$ be a pair of operators on a Hilbert space 
$\mathcal{H}$ such that $T$ is a contraction and $V$ is an isometry 
and also $T^*V=UVT^*$, where $U$ is unitary and $T,V \in \{U\}'$. 
Let $S$ on $\mathcal{K}$ be the minimal isometric dilation for $T$. 
If $\widetilde{V}$ on $\mathcal{K}$ is an isometric extension of $V$, 
then $( S,\widetilde{V} )$ is a pair of doubly twisted isometries 
on $\clk$ and hence there is a Wold-type decomposition for the pair 
$(S, \widetilde{V} )$ on the minimal space $\clk$. 
Moreover, the pair $(S, \widetilde{V} ) $ on the minimal space $\clk$ 
is unique up to unitary equivalence.
\end{thm}

\begin{proof} 
Suppose that $\left( T, V \right)$ is a pair consisting of a 
contraction $T$ and an isometry $V$ on $\mathcal{H}$. Also 
$T^*V=UVT^*$, where $U$ is unitary and $T,V \in \{U\}'$.
Thus the pairs $(T, U)$ and $(V, U)$ are doubly commuting on 
$\mathcal{H}$.

Let $S$ on $\mathcal{K}$ be the minimal isometric dilation
of $T$. Then 
\[
  \mathcal{K}= \overline{span} \left\lbrace S^nh:h \in \mathcal{H},
n \in \mathbb{Z}_+ \right\rbrace,  
\]
$S^*(\mathcal{H}) \subseteq \mathcal{H}$ and
$ T^*= S^*|_{\mathcal{H}}$. 
Also $TP_{\mathcal{H}}=P_{\mathcal{H}}S $ 
where $P_{\mathcal{H}}$ is the orthogonal projection from 
$\mathcal{K}$ onto $\mathcal{H}$. Consider
\[
\cll = \mbox{span}\{ S^n h_n: n \in \mathbb{Z}_+, h_n \in \clh \}.
\]

Define $\widetilde{V}$ on $\mathcal{L}$ as
\[
\widetilde{V} \left( \sum_{n=0}^{N} \alpha_n S^n h_n \right)
	= \sum_{n=0}^{N}\alpha_n S^n U^n V h_n, 
\]
where $\alpha_n \in \mathbb{C}$, $h_n \in \mathcal{H}$.
We shall firstly show that the map $\widetilde{V}$ is well defined on $\cll$. 
To do that we consider	the following. 
	
Let $\alpha_n, \beta_m \in \mathbb{C}$, and $h_n, g_m \in \mathcal{H}$. Then
\begin{align}\label{mainlemma_equa1}
\langle \sum_{n=0}^{N} \alpha_n S^n U^n Vh_n, \sum_{m=0}^{M} \beta_m S^m U^m V g_m \rangle	
& = \sum_{n=0}^N \sum_{m=0}^M \alpha_{n} \bar{ \beta}_{m}  \langle  S^n U^nVh_n, S^m U^mV g_m  \rangle. 
\end{align}	
Suppose $n \geq m$ for fixed $m, n$. Using the fact that $V$ is an 
isometry on $\clh$, $S$ is an isometric dilation of $T$, $T^*U=UT^*$ and
$T^*V =UV T^*$, we obtain
\begin{align*}
\langle  S^n U^n Vh_n, S^m U^m V g_m \rangle
&  = \langle  S^{n-m}U^n Vh_n, U^m  V g_m \rangle \\
& = \langle U^n Vh_n,  S^{*(n-m)} U^m V g_m \rangle \\
&=\langle U^n Vh_n,  T^{*(n-m)} U^m V g_m \rangle \\
&= \langle U^n Vh_n,  U^m T^{*(n-m)}V g_m \rangle \\
&= \langle U^{n-m} Vh_n,   U^{n-m}VT^{*(n-m)} g_m \rangle \\
& = \langle V h_n,  VT^{*(n-m)} g_m \rangle  \\
& = \langle h_n,  T^{*(n-m)} g_m \rangle  \\
& = \langle h_n,  S^{*(n-m)} g_m \rangle  \\
& = \langle S^{n-m} h_n,  g_m \rangle \\
& = \langle S^{n} h_n,  S^{m}g_m \rangle.
\end{align*}
Thus from the above equation (\ref{mainlemma_equa1}), we have
\begin{align*}
\langle \sum_{n=0}^{N} \alpha_n S^n U^n Vh_n, \sum_{m=0}^{M} \beta_m S^m U^m V g_m \rangle
&= \langle \sum_{n=0}^N  \alpha_{n}  S^{n} h_n,  \sum_{m=0}^M { \beta_{m}} S^{m}g_m \rangle. 
\end{align*}
In particular, we have
\begin{align}\label{mainlemma_equa2}
	\| \sum_{n=0}^{N} \alpha_n S^n U^n Vh_n \|^2 
	& = \| \sum_{n=0}^N  \alpha_{n}  S^{n} h_n \|^2. 
\end{align}
Suppose that
\[
\sum_{n=0}^N  \alpha_{n}  S^{n} h_n =  \sum_{m=0}^M { \beta_{m}} S^{m}g_m. 
\]	
Then from the above equality (\ref{mainlemma_equa2}), 
it is easy to see that
\[
\|  \sum_{n=0}^N  \alpha_{n}  S^{n} U^n V h_n -  \sum_{m=0}^M { \beta_{m}} S^{m} U^m  V g_m \|^2 
= \| \sum_{n=0}^N  \alpha_{n}  S^{n} h_n -  \sum_{m=0}^M { \beta_{m}} S^{m}g_m \|^2 =0. 
\]	
Therefore,
\[
\sum_{n=0}^N  \alpha_{n}  S^{n} U^n  V h_n =  \sum_{m=0}^M { \beta_{m}} S^{m} U^m  V g_m.
\]
Hence from the definition of $\widetilde{V}$, we have	
\[
\widetilde{V} \left( \sum_{n=0}^{N} \alpha_n S^n h_n \right)
= \sum_{n=0}^{N}\alpha_n S^n U^n V h_n. 
\]	
This proves that $\widetilde{V}$ is well defined on $\cll$. 
From equation (\ref{mainlemma_equa2}),
we can conclude that $\widetilde{V}$ is bounded as well as norm 
preserving linear operator on $\cll$. Again $\bar{\cll}=\clk$ 
follows that $\widetilde{V}$ is an isometry on $\clk$.
Also the definition of $\widetilde{V}$ on $\mathcal{K}$ 
implies that 
$\widetilde{V}(h_0)=V(h_0) $ for all $ h_0 \in \mathcal{H}$,
 i.e., $\widetilde{V}|_{\mathcal{H}}=V$. 
So $\widetilde{V}$ on $\mathcal{K}$ is an extension of $V$.

Now define $\widetilde{U}$ on $\mathcal{L}$ as
\[
\widetilde{U} \left( \sum_{n=0}^{N} a_n S^n h_n \right)
= \sum_{n=0}^{N}a_n S^n U h_n, 
\]
where $a_n \in \mathbb{C}$, $h_n \in \mathcal{H}$.
In the similar way as above, we can prove that $\widetilde{U}$ is well-defined 
and bounded as well as norm preserving operator on $\cll$. Since $U$ is unitary,
$\widetilde{U}$ is an isometric operator from $\cll$ onto $\cll$. Again $\bar{\cll}
=\clk$. Now using continuity we can show that $\widetilde{U}$ can be extended to 
a unitary operator (denoted by same $\widetilde{U}$) on $\clk$. From the definition 
of $\widetilde{U}$ on $\mathcal{K}$, we show  that 
$\widetilde{U}(h_0)=U(h_0) $ for all $ h_0 \in \mathcal{H}$,
i.e., $\widetilde{U}|_{\mathcal{H}}=U$.   

We shall now show that  $S$ and $\widetilde{V}$ are doubly 
twisted with respect to the twist $\widetilde{U}$ on $\mathcal{L}$.
Firstly,
\begin{align*}	
	S\widetilde{U} \left( \displaystyle{\sum_{n=0}^{N} \alpha_n S^n  h_n }\right) 
	= S \left( \displaystyle{\sum_{n=0}^{N}\alpha_n S^{n} U h_n }\right) 
	&= \displaystyle{\sum_{n=0}^{N} \alpha_n S^{n+1} U h_n }\\ 
	&=  \widetilde{U} \left( \displaystyle{\sum_{n=0}^{N} \alpha_n S^{n+1} h_n }\right)\\  
	&=\widetilde{U}S \left( \displaystyle{\sum_{n=0}^{N} \alpha_n S^{n} h_n }\right).
\end{align*}
Again,
 \begin{align*}	
     \widetilde{V}	\widetilde{U} \left( \displaystyle{\sum_{n=0}^{N} \alpha_n S^n  h_n }\right) 
	= \widetilde{V} \left( \displaystyle{\sum_{n=0}^{N}\alpha_n S^{n} U h_n }\right) 
	&= \displaystyle{\sum_{n=0}^{N} \alpha_n S^nU^n V U h_n }\\ 
	&= \displaystyle{\sum_{n=0}^{N} \alpha_n S^n UU^{n} V h_n }\\ 
	&=  \widetilde{U} \left( \displaystyle{\sum_{n=0}^{N} \alpha_n S^{n}U^n Vh_n }\right)\\  
	&=\widetilde{U}\widetilde{V} \left( \displaystyle{\sum_{n=0}^{N} \alpha_n S^{n} h_n }\right).
\end{align*}
Using the definition of $\widetilde{U}$ and $\widetilde{V}$, we have
\begin{align*}
 \widetilde{V}S \left( \displaystyle{\sum_{n=0}^{N} \alpha_n S^{n} h_n }\right)
 =	\widetilde{V} \left( \displaystyle{\sum_{n=0}^{N} \alpha_n S^{n+1} h_n }\right)
 &=\displaystyle{\sum_{n=0}^{N} \alpha_n S^{n+1} U^{n+1}  V h_n } \\
&=\widetilde{U} \left(\displaystyle{\sum_{n=0}^{N} \alpha_n S^{n+1} U^{n}  V h_n } \right)\\
&=\widetilde{U}S \left(\displaystyle{\sum_{n=0}^{N} \alpha_n S^{n} U^{n}  V h_n } \right) \\
& = \widetilde{U} S\widetilde{V} \left( \displaystyle{\sum_{n=0}^{N} \alpha_n S^n  h_n }\right).
\end{align*}	
Also
\begin{align*}
\widetilde{U} \widetilde{V}S^* \left( \displaystyle{\sum_{n=0}^{N} \alpha_n S^{n} h_n }\right) 
& = \widetilde{U} \widetilde{V} \left( \displaystyle{\sum_{n=0}^{N} \alpha_n S^* S^n h_n }\right)\\
& = \widetilde{U} \widetilde{V} \left( \alpha_0S^*h_0 +
		 \displaystyle{\sum_{n=1}^{N} \alpha_n S^{n-1} h_n }\right)\\
& =\widetilde{U} \widetilde{V} \left( \alpha_0T^*h_0 \right) +
\widetilde{U}\widetilde{V} \left( \displaystyle{\sum_{n=1}^{N} \alpha_n S^{n-1} h_n }\right)\\
& =\alpha_0 \widetilde{U} ( VT^*h_0)  + \widetilde{U}
\left( \displaystyle{\sum_{n=1}^{N} \alpha_n S^{n-1} U^{n-1} V h_n }\right)\\
& = \alpha_0 U VT^*h_0  + \left( 	\displaystyle{\sum_{n=1}^{N} \alpha_n S^{n-1} U U^{n-1} 
V h_n }\right)\\			
& =  \alpha_0 T^*Vh_0  +  S^*\left( \displaystyle{\sum_{n=1}^{N} \alpha_n S^{n} U^nV h_n }\right) \\
& =  S^* [\alpha_0 Vh_0  +  \left( \displaystyle{\sum_{n=1}^{N} \alpha_n S^{n}U^nV h_n }\right)]\\
& =  S^* \widetilde{V}[\alpha_0 S^{0} h_0  +  \left( \displaystyle{\sum_{n=1}^{N} \alpha_n S^{n} h_n }\right)]\\
& =  S^* \widetilde{V} \left( \displaystyle{\sum_{n=0}^{N} \alpha_n S^{n} h_n }\right). 
\end{align*}
This implies that  	
\[
 S\widetilde{U}=\widetilde{U}S, \quad \widetilde{V}\widetilde{U}=\widetilde{U}\widetilde{V}, 
 \quad  \widetilde{V}S=\widetilde{U}S\widetilde{V}  \quad \mbox{and}
  \quad   S^* \widetilde{V}= \widetilde{U}\widetilde{V}S^*
 \quad \mbox{on~} \cll.
\]		
Now the norm preserving operator $\widetilde{V}, \widetilde{U}$ on $\cll$
can be extended uniquely by continuity (again denoted by same 
$\widetilde{V}$, $\widetilde{U}$) to the closure of $\cll$ 
(i.e. $\bar{\cll}= \clk $ ) such that
\[
   S\widetilde{U}=\widetilde{U}S, \quad
    \widetilde{V}\widetilde{U}=\widetilde{U}\widetilde{V}, \quad 
     \widetilde{V}S=\widetilde{U}S\widetilde{V}  \quad \mbox{and} 
   \quad   S^* \widetilde{V}= \widetilde{U}\widetilde{V}S^*.
\]	
Therefore, the pair $(S, \widetilde{V})$ is a doubly twisted isometries 
with respect to the twist $\widetilde{U}$ on the minimal space $\mathcal{K}$. 
Hence, by Theorem \ref{Slocinski result-twist}, the doubly twisted pair 
of isometries $(S,\widetilde{V})$ admits Wold-type decomposition 
on $\mathcal{K}$.

Let $S_1$ and $S_2$ be two minimal isometric dilations on the space 
$\clk_1$ and  $\clk_2$ for the contraction $T$, respectively. Therefore,
\[
	\mathcal{K}_i= \overline{span} \left\lbrace S_i^nh:h \in \mathcal{H},
	n \in \mathbb{Z}_+ \right\rbrace , S_i^*(\clh) \subseteq \clh \ \text{and} 
	\ T^*=S_i^*|_{\clh} \ \ (i=1,2).
\]
Now, for $i=1,2$, consider $\cll_i = \mbox{span}\{ S_i^n h_n: n \in 
\mathbb{Z}_+, h_n \in \clh \}$ and define $\widetilde{V}_i$ on $\mathcal{L}_i$ as
\[
\widetilde{V}_i \left( \sum_{n=0}^{N_0} \alpha_n S_i^n h_n \right)
	= \sum_{n=0}^{N_0}\alpha_n S_i^n V h_n, 
\]
where $\alpha_n \in \mathbb{C}$, $h_n \in \mathcal{H}$.
Clearly, the maps $\widetilde{V}_i$ are bounded as well as norm 
preserving operators on $\cll_i$. Define a map 
$\hat{U}: \cll_1 \rightarrow \cll_2$ by
\[
    \hat{U} \left( \sum_{n=0}^{N} \alpha_n S_1^n h_n \right)
    = \sum_{n=0}^{N}\alpha_n S_2^n h_n. 
\] 
Using the fact that $T^*=S_i^*|_{\clh}$, it is easy to prove that the map 
$\hat{U}$ is well-defined and isometric linear map from $\cll_1$ onto $\cll_2$. Also
\begin{align*}
   \hat{U} \widetilde{V_1} \left( \sum_{n=0}^{N} \alpha_n S_1^n h_n \right)
   &=\hat{U}\left( \sum_{n=0}^{N} \alpha_n S_1^n U^n V h_n \right) \\
   &= \sum_{n=0}^{N} \alpha_n S_2^n  U^n Vh_n\\
   &= \widetilde{V_2} \left( \sum_{n=0}^{N} \alpha_n S_2^n h_n \right)\\
   &=\widetilde{V_2}\hat{U} \left( \sum_{n=0}^{N} \alpha_n S_1^n h_n \right).
\end{align*}
Therefore, $\hat{U}\widetilde{V_1} =\widetilde{V_2} \hat{U}$. Since
$\overline{\cll_i}=\clk_i$  for $i=1,2$; then by continuity $\hat{U}$ can 
be extended to a unitary operator (again denoted by same $\hat{U}$) from
$\clk_1$ to $\clk_2$. It says that the isometric extensions $\widetilde{V_1}$ 
and $\widetilde{V_2}$ of $V$ are isomorphic. Hence the pair $(S, \widetilde{V})$ 
is unique up to unitary equivalence on the minimal dilation space.
This completes the proof.
\end{proof}

The following results are immediate applications of the above Theorem \ref{Main result}. 
The result is known \cite{KO-WOLD TYPE} but here we provide a new and simple 
proof.

\begin{cor}\label{cor3}
Let $V$ be an isometry and $W$ be a coisometry on a Hilbert space $\clh$. 
If $(V,W)$ is twisted, then the pair $(V,W)$ is doubly twisted. 	
\end{cor}

\begin{proof}
Suppose that $V$ is an isometry, $W$ is a co-isometry and $U$ is a unitary 
on a Hilbert space $\clh$ satisfying $VW=UWV$ and $V,W \in \{U\}'$.
Let $S$ on $\mathcal{K} \supseteq \clh$ be the minimal isometric dilation of the 
co-isometry $W$, that is, 
\[
  \mathcal{K}= \overline{span} \left\lbrace S^nh:h \in \mathcal{H},
n \in \mathbb{Z}_+ \right\rbrace,  
\]
$S^*(\mathcal{H}) \subseteq \mathcal{H}$ and $W^*= S^*|_{\mathcal{H}}$. 
Now define $\widetilde{V}$ on $\cll = \mbox{span}\{ S^n h_n: n \in 
\mathbb{Z}_+, h_n \in \clh \}$ as
\[
\widetilde{V} \left( \sum_{n=0}^{N} \alpha_n S^n h_n \right)
= \sum_{n=0}^{N}\alpha_n S^n U^n V h_n
 \quad  \text{and} \quad 
\widetilde{U} \left( \sum_{n=0}^{N} \beta_n S^n h_n \right)
= \sum_{n=0}^{N}\beta_n S^n U h_n,
\]
where $\alpha_n, \beta_n\in \mathbb{C}$, $h_n \in \mathcal{H}$. 
Now in similar lines as Theorem \ref{Main result}, we can prove that 
$\widetilde{U}$ is unitary and $\widetilde{V}$ is isometry on $\clk$
with $\widetilde{V}S = \widetilde{U} S \widetilde{V}$ and 
$\widetilde{V} , S \in \{\widetilde{U}\}'$. Hence the pair $(\widetilde{V}, S)$ 
is a twisted isometries with a twist $\widetilde{U}$ on the minimal 
space $\mathcal{K}$.
	
Since $W$ is a co-isometry on $\clh$, the minimal isometric dilation $S$ 
of $W$ is unitary on $\clk$. Consequently, $S^*\widetilde{V}=\widetilde{U}
\widetilde{V} S^*$ which is same as $\widetilde{V}^*S=\widetilde{U}^*
S\widetilde{V}^*$. Using the facts $S^*|_{\clh}=W^*$, $\widetilde{U}|_{\clh}=U$  
and $\widetilde{V}|_{\clh}= V$, we have
\[
W^*Vh= S^*Vh=S^* \widetilde{V} h= \widetilde{U} \widetilde{V} S^* h=  
\widetilde{U}\widetilde{V} W^* h= \widetilde{U}VW^*h=UVW^*h ~~~~
\quad (h \in \clh). 
\] 
This implies $W^*V=UVW^*$ on $\clh$. Hence the pair $\left(V, W \right)$ 
is doubly twisted. 	
\end{proof}

\begin{rem}
It is note worthy to mention that we have used Theorem \ref{Main result}
to prove the above result. Alternatively, we provide a simple proof of the
result as follows: Suppose $(V,W)$ is a twisted pair with a twist $U$ on 
$\clh$ such that $V$ is an isometry and $W$ is a coisometry. Then we obtain
\begin{align*}
	&(W^*V-UVW^*)^*(W^*V-UVW^*)\\
	&=(V^*W-WV^*U^*)(W^*V-UVW^*) \\
	&=V^*WW^*V-U^*WV^*W^*V-UV^*WVW^* +U^*UWV^*VW^*\\
	&=I_{\clh}-U^*UWW^*V^*V-UU^*V^*VWW^*+I_{\clh}\\
	&=0.
\end{align*}
Consequently,
\[
W^*V-UVW^*=0, \quad \mbox{that is}, \quad V^*W=U^*WV^*.
\]
Hence the pair $(V,W)$ is doubly twisted on $\clh$. 
\end{rem}

In the following, we record the above result.
\begin{lem}\label{Doubly-twisted-isometries}
Let $(V_1, V_2)$ be a pair of isometries and $U$ be unitary on 
$\clh$ such that $V_1^*V_2=U^*V_2V_1^*$ with $V_1,V_2 \in \{U\}'$.
Then $(V_1, V_2)$ is doubly twisted isometries with a twist $U$. 
\end{lem}

Using Theorem \ref{Main result}, we can generalize the result 
Theorem 2.3 in \cite{BKS-CANONICAL} to a pair of twisted isometries. 
We only state the result as the proof is similar lines.

\begin{thm}\label{Twisted-finite dim wandering}
Let $(V_1, V_2)$ be a pair of twisted isometries on $\clh$ and  
$ \dim \ker {V_2^*}$ is finite. If $\mathcal{H}=\mathcal{H}_{u1} 
\oplus \mathcal{H}_{s1}$ is the Wold decomposition for $V_1$, then 
the decomposition reduces $V_2$.
\end{thm}

Suppose that $S$ on $\mathcal{K}$ is the minimal isometric dilation of 
a contraction $T$ on $\mathcal{H}$. Then the Wold decomposition of $S$ 
is $S_{u} \oplus S_{s}$ on $\clk=\clk_{u} \oplus \clk_{s}$ such that 
$S_u=S|_{\clk_u}$ is unitary and $S_{s}=S|_{\clk_s}$ is unilateral shift. 
Moreover, 
		\[
		\clk_u={\bigcap_{n \geq 0} S^n \clk}, \qquad
		\clk_s={\bigoplus_{n\geq 0} S^n\clw_S},
		\]
where $\clw_S=\ker S^*$ is the wandering subspace for $S$ defined by 
$\overline{(I-ST^*)\mathcal{H}}$ (cf. \cite{FF-BOOK}). Also defect 
index $\sigma_{T^*}$ for a contraction $T$ on $\clh$ is defined as 
$\dim \overline {D_{T^*}\clh}$, where $D_{T^*}=(I-TT^*)^\frac{1}{2}$.  
If $\sigma_{T^*}$ is  finite, then we have the following result:

\begin{prop}
Let $\left( T, V \right)$ be a pair of twisted operators on 
$\mathcal{H}$ such that $T$ is a contraction and $V$ is an isometry. 
Let $S$ on $\mathcal{K}$ 
be the minimal isometric dilation for $T$ and $\widetilde{V}$ 
on $\mathcal{K}$ is an isometric extension of $V$. If $ \sigma_{T^*}$ 
is finite and $\widetilde{V}_{u} \oplus \widetilde{V}_{s}$ on 
$\widetilde{\clk_{u}} \oplus \widetilde{\clk_{s}}$ is the Wold 
decomposition for isometry $\widetilde{V}$ on $\clk$, then both 
$\widetilde{\clk_{u}}$ and $\widetilde{\clk_{s}}$ reduce $S$.
\end{prop}

\begin{proof}
Assume that $\left( T, V \right)$ is a pair of twisted operators with
a twist $U$ on $\mathcal{H}$ such that $T$ is a contraction and $V$ 
is an isometry.

Let $S$ on $\mathcal{K}$ be the minimal isometric dilation of $T$. Therefore, 
\[
\mathcal{K}= \overline{span} \left\lbrace S^nh:h \in \mathcal{H},
		n \in \mathbb{Z}_+ \right\rbrace,  
\]
$S^*(\mathcal{H}) \subseteq \mathcal{H}$ and $ T^*= S^*|_{\mathcal{H}}$. 		
Consider $\cll = \mbox{span}\{ S^n h_n: n \in \mathbb{Z}_+, h_n \in \clh \}$.
Now define $\widetilde{V}$ and $\widetilde{U}$ on $\cll$ as
\[
\widetilde{V} \left( \sum_{n=0}^{N} \alpha_n S^n h_n \right)
= \sum_{n=0}^{N}\alpha_n S^n U^n V h_n
\quad  \text{and} \quad 
\widetilde{U} \left( \sum_{n=0}^{N} \beta_n S^n h_n \right)
= \sum_{n=0}^{N}\beta_n S^n U h_n,
\]
where $\alpha_n, \beta_n\in \mathbb{C}$, $h_n \in \mathcal{H}$. 
It is now easy to check that the map $\widetilde{V},\widetilde{U}$ is 
well-defined and bounded as well as norm  preserving linear operator 
on $\cll$. Again by continuity $\widetilde{V}$ and the untary $\widetilde{U}$ 
can be extended to $\bar{\cll}=\clk$ (write same $\widetilde{V}, \widetilde{U}$ 
on $\clk$). Applying Theorem \ref{Main result}, we have the pair $(S, \widetilde{V})$ 
is a twisted isometries with a twist $\widetilde{U}$ on $\mathcal{K}$.

Now for every $h \in \clh$
\[
\|(I-ST^*)h \|^2= \|h\|^2 -2 \ \text{Re} < h, ST^*h >+ \|T^*h\|^2
=\|h \|^2 -\|T^*h\|^2= \|D_{T^*}h\|^2.
\]
Define $\hat{U}:{D_{T^*}\clh} \rightarrow (I-ST^*)\clh$ as
\[
\hat{U}({D_{T^*}h}) = (I-ST^*)h.
\]		
Then $\hat{U}$ is an isometry from ${D_{T^*}\clh}$ onto $(I-ST^*)\clh$.
Hence $(I-ST^*)\clh$ is finite dimensional as by hypothesis 
$\sigma_{T^*}= \dim \overline {D_{T^*}\clh}$ is finite.  
Therefore, $\hat{U}$ is unitary and $\sigma_{T^*}= \dim(I-ST^*)\clh
= \dim \ker S^*$.

Suppose $\widetilde{V}_{u} \oplus \widetilde{V}_{s}$ on 
$\widetilde{\clk_{u}} \oplus \widetilde{\clk_{s}}$ is the Wold 
decomposition for isometry $\widetilde{V}$ on $\clk$. Since 
$( \widetilde{V}, S)$ is a pair of twisted isometries on 
$\mathcal{K}$ and $\dim \ker S^*$ is finite dimensional, from 
Theorem \ref{Twisted-finite dim wandering} we conclude that the  
decomposition reduces the isometric dilation $S$.

This completes the proof.		
\end{proof}

\section{Doubly twisted Isometries}

In \cite{MSS-PAIRS} the second author, Sarkar and Sankar characterize the
pair of doubly commuting isometries and also studied defect operator. 
In this section, we characterize pair of doubly twisted isometries.

Let $(V_1, V_2)$ be a pair of twisted isometries on a Hilbert space 
$\mathcal{H}$. Throughout this section we will use the following notations:
\begin{align*}
V & =V_1V_2 \\
\clw & =\clw(V)= \clh \ominus V_1V_2\clh,\\
\text{and} \quad \clw_i & =\clw(V_i)=\clh \ominus V_i \clh 
\quad \text{for }\quad i=1,2.
\end{align*}

The following relations are useful to our discussion. 
\begin{lem}\label{Lemma properties of twisted isometries}
Let $(V_1,V_2)$ be a pair of twisted isometries with respect to a twist 
$U$ on a Hilbert space $\clh$ and let $V=V_1V_2$. Then for $n \in \mathbb{N}$, 
we have the following:
\begin{enumerate}
\item $V_1V_2^n=U^n V_2^nV_1$ and $V_2V_1^n=U^{*n}V_1^nV_2$,
\item $V_1V^n=U^nV^nV_1$ and $V_2 V^n=U^{*n}V^n V_2$,
\item  $V_1^*V^n=U^{*(n-1)}V^{n-1}V_2$ and $V_2^*V^n= U^nV^{n-1} V_1$,
\item $V^n= U^{\frac{n(n+1)}{2}} V_2^n V_1^n=U^{*{\frac{n(n-1)}{2}}}V_1^nV_2^n$.
\end{enumerate}
\end{lem}

\begin{proof}
Suppose $(V_1,V_2)$ is a pair of twisted isometries with respect to a twist 
$U$ on $\clh$. Then $V_1V_2= UV_2V_1$, $UV_1=V_1U$ and $UV_2=V_2U$, where 
$U$ is unitary on $\clh$. Using this relation, we can prove the all the 
equations. 
\end{proof}

\begin{lem}\label{Twisted isometries lemma}
Let $(V_1, V_2)$ be a pair of twisted isometries on a Hilbert space 
$\mathcal{H}$. Then 
\[
\mathcal{W}=\mathcal{W}_1 \oplus V_1 \mathcal{W}_2=V_2 \mathcal{W}_1 
\oplus \mathcal{W}_2, 
\] 	
and the operator $\hat{U}$ on $\mathcal{W}$ defined by
\[
\hat{U}(\zeta_1 \oplus V_1 \zeta_2)= V_2\zeta_1 \oplus \zeta_2 
\]
for $\zeta_1 \in \mathcal{W}_1$ \text{and} $\zeta_2 \in \mathcal{W}_2$, 
is a unitary operator. Moreover, 
\[
V_i^* \clw= \clw_j \quad \mbox{for} \quad 1\leq i, j \leq 2 \mbox{~and}
~i \neq j.
\]
\end{lem}

\begin{proof}
Suppose $(V_1, V_2)$ is a pair of twisted isometries with respect to a twist $U$
on $\mathcal{H}$. Then $V_1V_2= UV_2V_1$, and $V_1, V_2 \in \{ U \}'$. Since $U$ 
is unitary, the pairs $(V_i, U)$ are doubly commuting on $\clh$ for $i=1, 2$.
Let $V=V_1V_2$. Then
\[
I -VV^*= (I-V_1V_1^*) \oplus V_1(I-V_2V_2^*)V_1^*=V_2(I-V_1V_1^*)V_2^* \oplus (I-V_2V_2^*). 
\]	
This implies that
\[
\mathcal{W}=\mathcal{W}_1 \oplus V_1 \mathcal{W}_2= V_2 \mathcal{W}_1 \oplus \mathcal{W}_2.
\] 
Now the operator $\hat{U}$ is unitary directly follows from the above line.

For the last part, we have
\[
	V_i^* \clw= V_i^*(V_i\clw_j \oplus \clw_i)=\clw_j
\]
for $i \neq j$.	
\end{proof}

We present an important result and will be used to characterize the 
doubly twisted isometries.

\begin{lem}\label{Joint reducing subspace}
Let $(V_1,V_2)$ be a pair of twisted isometries with respect to a twist 
$U$ on $\clh$ and let $V=V_1V_2$. Then $\clh_s(V)$ and $\clh_u(V)$ are 
joint $(V_1, V_2)$-reducing subspaces of $\clh$.  
Moreover,
\[
\clh_u(V) \subseteq \clh_u(V_i) \quad   \text{and} \quad  \clh_s(V_i) 
\subseteq \clh_u(V)
\]
for $i=1,2$.
\end{lem}
\begin{proof}
Suppose that $(V_1, V_2)$ is a pair of twisted isometries with respect 
to a twist $U$ on $\mathcal{H}$. Then $V_1V_2= UV_2V_1$, and $UV_i=V_iU$ 
for $i=1,2$. From Lemma \ref{Twisted isometries lemma}, we have 
\begin{align*}
V_1\clw=V_1\clw_2 \oplus V\clw_1
&= V_1(\clh \ominus V_2\clh) \oplus V(\clh \ominus V\clh)\\
& \subseteq (\clh \ominus V\clh) \oplus V\clw \\
&= \clw \oplus V\clw.
\end{align*}
Now $(U, V_i)$ is a doubly commuting pair and hence $U\clw_i = \clw_i$
for $i=1,2$. Therefore, by Lemma \ref{Lemma properties of twisted isometries}, 
for each $n \geq 0$ we have
\begin{equation}\label{DO Equation-1}
V_1V^n \clw =U^n V^nV_1 \clw \subseteq V^n(\clw \oplus V\clw) \subseteq \clh_s(V).
\end{equation}
Similarly, 
\begin{align*}
&V_2\clw=V_2\clw_1 \oplus UV\clw_2=V_2\clw_1 \oplus V\clw_2 \subseteq
\clw \oplus V\clw,
\end{align*}
which follows that, for $n \geq 0$,
\begin{align}\label{DO Equation-2}
V_2V^n \clw = V^nV_2U^{*n} \clw=V^n V_2 \clw \subseteq V^n(\clw 
\oplus V\clw) \subseteq \clh_s(V).
\end{align}
From the equations (\ref{DO Equation-1}) and (\ref{DO Equation-2}), 
it follows that $\clh_s(V)$ is joint $(V_1,V_2)$-invariant subspace.
	
\NI
Again using Lemma \ref{Lemma properties of twisted isometries}, we have 
\[
V_1^* V^n \clw= U^{*{n-1}}V^{n-1}V_2 \clw = V^{n-1}(V_2 \clw) \subseteq 
V^{n-1}(\clw \oplus V\clw), 
\]
and 
\[
V_2^*V^n \clw= U^nV^{n-1}(V_1\clw)= V^{n-1}(V_1\clw)= V^{n-1}(\clw \oplus V\clw).
\]
Thus $\clh_s(V)$ is $(V_1^*, V_2^*)$-invariant subspace. Hence $\clh_s(V)$ 
is a joint $(V_1,V_2)$ reducing subspace of $\clh$. 
	
Now $\clh_u(V)= \clh_s(V)^{\perp}$, and hence $\clh_u(V)$ and $\clh_s(V)$ 
are both joint $(V_1,V_2)$-reducing subspaces of $\clh$.  
From Lemma \ref{Lemma properties of twisted isometries}, we have for $n \geq 0$,
\begin{align*}
V^n \clh&=U^{\frac{n(n+1)}{2}} V_2^n V_1^n\clh = V_2^n (V_1^n \clh) \\
&=U^{*{\frac{n(n-1)}{2}}}V_1^nV_2^n\clh=V_1^n (V_2^n \clh),
\end{align*}
which implies
\[
V^n \clh \subseteq V_1^n \clh, V_2^n \clh. 
\]
Consequently, $\clh_u(V) \subset \clh_u(V_i)$ and hence 
$\clh_u(V_i)^{\perp} \subseteq \clh_u(V)^{\perp}$. Therefore
$ \clh_s(V_i) \subseteq \clh_u(V)$ for $i =1,2$.
This completes the proof.
\end{proof}

The following result is a characterization for doubly twisted isometries.
  
\begin{thm}
Let $(V_1, V_2)$ be a pair of twisted isometries with respect to a twist 
$U$ on a Hilbert space $\clh$. Then the followings are equivalent:
\begin{enumerate}
\item $(V_1,V_2)$ is doubly twisted with respect the same twist $U$
\item $V_1\clw_2 \subset \clw_2$ 
\item $V_2\clw_1 \subseteq \clw_1$. 
\end{enumerate}
\end{thm}	

\begin{proof}
Suppose that $(V_1, V_2)$ is a pair of doubly twisted with respect 
to a twist $U$ on $\clh$. Then 
\[
V_1V_2=UV_2V_1, \quad V_1^{*}V_2=U^*V_2V_1^{*}, \quad \text{and} 
\quad V_1, V_2 \in \{U\}^{'}.
\] 
Now
\[
V_1(I-V_2V_2^*)V_1^*= (V_1-UV_2V_1V_2^*)V_1^*=(V_1-UV_2U^*V_2^*V_1)
V_1^*=(I-V_2V_2^*)V_1V_1^*.
\]
It follows $V_1\clw_2 \subseteq \clw_2$, i.e., (1) implies (2). 
Similarly, we can prove (1) implies (3).

Now we shall only prove that (2) implies (1). Since the pair $(V_1, V_2)$ 
is twisted isometries with respect to a twist $U$,
\[
V_1V_2=UV_2V_1, \quad UV_1=V_1U, \quad \text{and}\quad UV_2=V_2U.
\] 
Consider the Wold-von Neumann decomposition for the isometry $V=V_1 V_2$: 
\[
\mathcal{H}=\mathcal{H}_s(V) \oplus \mathcal{H}_u (V). 
\]
Now Lemma \ref{Joint reducing subspace} gives that $\mathcal{H}_s (V), 
\mathcal{H}_u(V)$ are $V_1, V_2$-reducing subspaces and hence $V_i|_{\clh_u(V)}$ 
is unitary for $i=1,2$. Again $(U,V)$ is doubly commuting, so by Corollary 
\ref{DC-reduce}, both the subspaces $\clh_u(V)$ and $\clh_s(V)$ reduce $U$. 
Therefore, $(V_1|_{\mathcal{H}_u(V)}), V_2|_{\mathcal{H}_u(V)})$ on 
$\mathcal{H}_u$ is doubly twisted with respect to the twist $U|_{\clh_u(V)}$. 
It remains to prove that $V_2^*V_1=UV_1V_2^*$ on $\clh_s(V)$. Now
\begin{align*}
(V_2^*V_1-UV_1V_2^*)V^n
&=V_2^*V_1V^n -UV_1V_2^*V^n\\
&= V_2^*V_1V_1V_2V^{n-1}-UV_1V_2^*V_1V_2V^{n-1}\\
&=U^2 V_1^2 V^{n-1}- U^2 V_1^2 V^{n-1} \\
&=0,
\end{align*}
and hence
\[
V_2^*V_1- UV_1V_2^*=0 \quad \text{on} \quad  V^n \clw \quad   
\forall \quad  n \geq 1.
\]
To complete the proof we shall show that 
\[
V_2^*V_1=UV_1 V_2^* \quad \text{on} \quad \mathcal{W}.
\] 
Let $\zeta \in \clw$. Then By Lemma \ref{Twisted isometries lemma}, 
\[
\zeta=V_2\zeta_1 \oplus \zeta_2  
\] 
for some $\zeta_1 \in \clw_1=\ker V_1^*$ and $\zeta_2 \in \clw_2=\ker V_2^*$.
Since $V_1 \clw_2 \subseteq \clw_2$, 
\begin{align*}
V_2^*V_1 \zeta = V_2^*V_1(V_2 \zeta_1 \oplus \zeta_2 )
&= V_2^* V_1V_2\zeta_1 +  V_2^* (V_1\zeta_2)\\
&=UV_2^* V_2 V_1 \zeta_1 \\
&=UV_1 \zeta_1.
\end{align*}
Again
\begin{align*}
UV_1 V_2^* \zeta= UV_1 V_2^*(V_2 \zeta_1 \oplus \zeta_2 )
&=UV_1 V_2^*V_2\zeta_1 + UV_1 V_2^* \zeta_2 \\
&=UV_1 \zeta_1.
\end{align*}
Hence $(V_1, V_2) $ is a pair of doubly twisted isometries on $\clh$. 
\end{proof}	

\begin{rem}
It is noted that we can characterize doubly twisted isometries by 
studying defect operators. Characterization for doubly twisted 
isometries and more on defect operators will be carried out in our 
future papers.
\end{rem}

\label{Ref}

\end{document}